\documentclass[twoside]{article}

\usepackage{vmargin}            % Régler la taille de la feuille.
\usepackage{fancyhdr}           % Régler le titre courant et le bas de page.
\usepackage{calc}               % Faire des calculs sur les longueurs.
\usepackage{ifthen}             % Faire des tests if/then/else.
\usepackage{pifont}             % La police \ding.
\usepackage{supertabular}       % Les grands tableaux.
\usepackage{multicol}           % Plusieurs colonnes.
\usepackage{xspace}             % Ajuster l'espace après des mots.

\usepackage{amsmath}    % Les symboles les plus fréquents.
\usepackage{amssymb}    % Des symboles.
\usepackage{amsfonts}   % Des fontes, eg pour \mathbb.
\usepackage{verbatim}   % Pour les codes sources en informatique.
\usepackage{mathrsfs}   % Des lettres majuscules cursives (\mathscr).
\usepackage{dsfont}

\usepackage{array}
\usepackage{curves}
\usepackage{epic}
\usepackage{eepic}
\usepackage{epsfig}
\usepackage{graphics}
\graphicspath{{PS/}}

\setpapersize{A4}
\setmarginsrb{33mm}{20mm}{33mm}{25mm}{10mm}{6mm}{0mm}{10mm}

\newcounter{ifnote}
\newcounter{Corrige}
\newcounter{Equation}[Corrige]

\catcode`\«=\active
\catcode`\»=\active
\def«{\og\ignorespaces}
\def»{{\fg}}

        {\end{filet}\bigskip}

\newenvironment{custom-itemize}[1]%
        {%
        \begin{list}{}%
                {%
                \settowidth{\labelwidth}{#1}%
                \setlength{\leftmargin}{\labelwidth+\labelsep}%
                }%
        }%
        {\end{list}}
        {\end{custom-itemize}}

\reversemarginpar \marginparwidth 2.6cm

\newcounter{Note}[page]
\newcounter{notesimple}
\newcounter{notemark}
\newcounter{notemarkref}
\newcounter{notetext}
\newcounter{notetextref}
\newcommand{\notemark}{%
        \ifthenelse{\equal{\value{notemark}}{\value{notemarkref}}}{%
                \refstepcounter{notemark}%
                \refstepcounter{notemarkref}%
                \setcounter{Note}{\value{notemark}}%
       }{%
                \refstepcounter{notemark}%
                \setcounter{Note}{\value{notemark}}%
        }%
        \setcounter{notetextref}{\value{notemarkref}}%
        \begin{picture}(0,0)%
                \put(-3,-3){\LARGE\theNote}
        \end{picture}\xspace
}

\newcommand{\notetextbase}[1]{%
        \ifthenelse{\value{ifnote} = 1}{%
        \marginpar{\notedebasenumero{#1}}%
        }{}%
}

\newcommand{\notetext}[1]{
        \ifthenelse{\equal{\value{notetext}}{\value{notetextref}}}{%
                \refstepcounter{notetext}%
                \refstepcounter{notetextref}%
                \setcounter{Note}{\value{notetext}}
        }{
                \refstepcounter{notetext}%
                \setcounter{Note}{\value{notetext}}
        }%
\setcounter{notemarkref}{\value{notetextref}}%
\notetextbase{#1}%
}

\newcommand{\notedebase}[1]{%
        \raggedright%
        \footnotesize%
        \vskip-0.5\baselineskip
        \rule[-1.4mm]{\linewidth}{0.5pt}
        \rule[1.4mm]{\linewidth}{0.5pt} \\%
        \vskip-0.5\baselineskip
        #1 \\%
        \vskip-\baselineskip
        \rule[-1.4mm]{\linewidth}{0.5pt}
        \rule[1.4mm]{\linewidth}{0.5pt} \\%
}

\newcommand{\notedebasenumero}[1]{%
        {\LARGE\theNote}
        \notedebase{#1}%
}

\newcommand{\notesimple}[1]{%
        \ifthenelse{\value{ifnote} = 1}{%
        \marginpar{\notedebase{#1}}%
        }{}%
}

\newcommand{\notenumero}[1]{%
        \ifthenelse{\value{ifnote} = 1}{%
        \refstepcounter{Note}%
        \setcounter{notemarkref}{\value{Note}}%
        \setcounter{notetextref}{\value{Note}}%
        \setcounter{notemark}{\value{Note}}%
        \setcounter{notetext}{\value{Note}}%
        \begin{picture}(0,0)%
        \put(-3,-3){\LARGE\theNote}
        \end{picture}%
        \ifthenelse{\isodd{\value{Note}}}{%
                \protect\reversemarginpar%
                \marginpar[{\notedebasenumero{#1}}]{\notedebasenumero{#1}}}{%
                \protect\normalmarginpar%
                \marginpar[{\notedebasenumero{#1}}]{\notedebasenumero{#1}}}%
        }{}%
}

\newcommand{\note}[2][1]{%
        \ifthenelse{\equal{#1}{0} \or \equal{#1}{simple}}{%
                \setcounter{notesimple}{1}%
        }{}%
        \ifthenelse{\value{notesimple} = 0}{%
                \notenumero{#2}%
        }{%
                \notesimple{#2}%
        }
}

\newcommand{\sk}{\smallskip}
\newcommand{\mk}{\medskip}
\newcommand{\bk}{\bigskip}

\newcounter{cinquantaine}
\newcounter{centaine}
\newcommand{\reperes}[1][20]{
        \ifthenelse{\value{ifnote} = 1}{%
                \setcounter{cinquantaine}{#1/5+1}%
                \setcounter{centaine}{#1/10+1}%
                \multiput(0,0)(10,0){#1}{\line(0,1){1}}%
                \multiput(0,0)(10,0){#1}{\line(0,-1){1}}%
                \multiput(0,0)(0,10){#1}{\line(1,0){1}}%
                \multiput(0,0)(0,10){#1}{\line(-1,0){1}}%
                \multiput(0,0)(-10,0){#1}{\line(0,1){1}}%
                \multiput(0,0)(-10,0){#1}{\line(0,-1){1}}%
                \multiput(0,0)(0,-10){#1}{\line(1,0){1}}%
                \multiput(0,0)(0,-10){#1}{\line(-1,0){1}}%
                \multiput(0,0)(50,0){\thecinquantaine}{\line(0,1){3}}%
                \multiput(0,0)(50,0){\thecinquantaine}{\line(0,-1){3}}%
                \multiput(0,0)(0,50){\thecinquantaine}{\line(1,0){3}}%
                \multiput(0,0)(0,50){\thecinquantaine}{\line(-1,0){3}}%
                \multiput(0,0)(-50,0){\thecinquantaine}{\line(0,1){3}}%
                \multiput(0,0)(-50,0){\thecinquantaine}{\line(0,-1){3}}%
                \multiput(0,0)(0,-50){\thecinquantaine}{\line(1,0){3}}%
                \multiput(0,0)(0,-50){\thecinquantaine}{\line(-1,0){3}}%
                \multiput(0,0)(100,0){\thecentaine}{\line(0,1){5}}%
                \multiput(0,0)(100,0){\thecentaine}{\line(0,-1){5}}%
                \multiput(0,0)(0,100){\thecentaine}{\line(1,0){5}}%
                \multiput(0,0)(0,100){\thecentaine}{\line(-1,0){5}}%
                \multiput(0,0)(-100,0){\thecentaine}{\line(0,1){5}}%
                \multiput(0,0)(-100,0){\thecentaine}{\line(0,-1){5}}%
                \multiput(0,0)(0,-100){\thecentaine}{\line(1,0){5}}%
                \multiput(0,0)(0,-100){\thecentaine}{\line(-1,0){5}}%
        }%
        {}%
}

\newcommand{\mathgras}[1]{\ensuremath{%
        {\text{\mathversion{bold}\ensuremath{#1}}}%
        }}

\def\leaderfill{\leaders\hbox to 1ex{\hss.\hss}\hfill}

        {\setlength{\columnsep}{8 mm}%
        \setlength{\columnseprule}{0.5 pt}%
        \begin{multicols}{2}}%
        {\end{multicols}}

\makeatletter

\newbox\bk@bxb
\newbox\bk@bxa
\newif\if@bkcont
\newif\ifbkcount
\newcount\bk@lcnt

\def\breakboxskip{2pt}
\def\breakboxparindent{1.8em}
\def\margesep{1cm}      % Écart entre la marge de gauche et le filet.
\def\intervalle{1mm}    % Écart supplémentaire entre le filet et le texte.

\def\filet{\vskip\breakboxskip\relax
\setbox\bk@bxb\vbox\bgroup
\advance\linewidth -\fboxrule
\advance\linewidth -\margesep
\advance\linewidth -\intervalle
\advance\linewidth -\fboxsep
\hsize\linewidth\@parboxrestore
\parindent\breakboxparindent\relax}

\def\bk@split{%
\@tempdimb\ht\bk@bxb % height of original box
\advance\@tempdimb\dp\bk@bxb
\setbox\bk@bxa\vsplit\bk@bxb to\z@ % split it
\setbox\bk@bxa\vbox{\unvbox\bk@bxa}% recover height & depth of \bk@bxa
\setbox\@tempboxa\vbox{\copy\bk@bxa\copy\bk@bxb}% naive concatenation
\advance\@tempdimb-\ht\@tempboxa
\advance\@tempdimb-\dp\@tempboxa}% gap between two boxes

\def\bk@addfsepht{%
     \setbox\bk@bxa\vbox{\vskip\fboxsep\box\bk@bxa}}

\def\bk@addskipht{%
     \setbox\bk@bxa\vbox{\vskip\@tempdimb\box\bk@bxa}}

\def\bk@addfsepdp{%
     \@tempdima\dp\bk@bxa
     \dp\bk@bxa\@tempdima}

\def\bk@addskipdp{%
     \@tempdima\dp\bk@bxa
     \advance\@tempdima\@tempdimb
     \dp\bk@bxa\@tempdima}

\def\bk@line{%
    \hbox to \linewidth{\ifbkcount\smash{\llap{\the\bk@lcnt\ }}\fi
    \hskip\margesep
    \vrule \@width\fboxrule\hskip\fboxsep
    \hskip\intervalle
    \box\bk@bxa\hfil
        }}

\def\endfilet{\egroup
\ifhmode\par\fi{\noindent\bk@lcnt\@ne
\@bkconttrue\baselineskip\z@\lineskiplimit\z@
\lineskip\z@\vfuzz\maxdimen
\bk@split\bk@addfsepht\bk@addskipdp
\ifvoid\bk@bxb      % Only one line
\def\bk@fstln{\bk@addfsepdp
\vbox{\bk@line}}%
\else               % More than one line
\def\bk@fstln{\vbox{\bk@line}\hfil
\advance\bk@lcnt\@ne
\loop
 \bk@split\bk@addskipdp\leavevmode
\ifvoid\bk@bxb      % The last line
 \@bkcontfalse\bk@addfsepdp
 \vtop{\bk@line}%
\else               % 2,...,(n-1)
 \bk@line
\fi
 \hfil\advance\bk@lcnt\@ne
\if@bkcont\repeat}%
\fi
\leavevmode\bk@fstln\par}\vskip\breakboxskip\relax}

\bkcountfalse

\makeatother

\newlength{\leftlength}
\newlength{\rightlength}
\newlength{\calculskip}
\newcommand{\calculvskip}[1]{%
  \ifthenelse{#1 = 0}{\setlength{\calculskip}{0pt}}{}%
  \ifthenelse{#1 = 1}{\setlength{\calculskip}{\smallskipamount}}{}%
  \ifthenelse{#1 = 2}{\setlength{\calculskip}{\medskipamount}}{}%
  \ifthenelse{#1 = 3}{\setlength{\calculskip}{\bigskipamount}}{}%
  \ifthenelse{#1 = 4}{\setlength{\calculskip}{1cm}}{}%
  \vskip\calculskip
}

\newcommand{\leftcentersright}[4][2]{%
        \settowidth{\leftlength}{#2}%
        \settowidth{\rightlength}{#4}%
        \calculvskip{#1}
        \noindent#2\hskip-\leftlength%
        \hfill#3\hfill
        \mbox{}\hskip-\rightlength#4%
        \vskip\calculskip%
        }

\newcommand{\centers}[2][2]{\leftcentersright[#1]{}{#2}{}}
\newcommand{\leftcenters}[3][2]{\leftcentersright[#1]{#2}{#3}{}}
\newcommand{\centersright}[3][2]{\leftcentersright[#1]{}{#2}{#3}}

\newsavebox{\boite}
\def\debutcom{\begin{lrbox}{\boite}}
\def\fincomg{\end{lrbox}\makebox[0cm][l]{\usebox{\boite}}%
             \hskip\linewidth\hskip-\textwidth}
\def\fincomd{\end{lrbox}\makebox[0cm][r]{\usebox{\boite}}}

\newenvironment{calculs:base}[2][2]{%
        \calculvskip{#1}
        \noindent
        \begin{tabular*}{\linewidth}[t]%
                {@{}>{\debutcom}l<{\fincomg}@{\extracolsep{\fill}}%
                >{$}r<{$}%
                @{$\ #2\ $}%
                @{\extracolsep{0pt}}>{$}l<{$}%
                @{\extracolsep{\fill}}>{\debutcom}r<{\fincomd}@{}}
        }{%
        \end{tabular*}%
        \vskip\calculskip
}

\newenvironment{inegalites:leq}[1][2]{%
        \begin{calculs:base}[#1]{\leq}}{%
        \end{calculs:base}
}

\newenvironment{inegalites:geq}[1][2]{%
        \begin{calculs:base}[#1]{\geq}}{%
        \end{calculs:base}
}

        {\begin{calculs:rcl}[#1]{r}{c}{l}}%
        {\end{calculs:rcl}}

\newenvironment{calculs:rcl}[4][2]{%
        \calculvskip{#1}
        \noindent
        \begin{tabular*}{\linewidth}[t]%
                {@{}>{\debutcom}l<{\fincomg}@{\extracolsep{\fill}}
                >{$}#2<{$}@{\extracolsep{0pt}}%
                >{$\ }#3<{\ $}%
                @{\extracolsep{0pt}}>{$}#4<{$}%
                @{\extracolsep{\fill}}>{\debutcom}r<{\fincomd}@{}}%
        }{%
        \end{tabular*}%
        \vskip\calculskip
}

\newenvironment{calculs:rcl:extracol}[5][2]{%
        \calculvskip{#1}
        \noindent
        \begin{tabular*}{\linewidth}[t]%
                {@{}>{\debutcom}l<{\fincomg}@{\extracolsep{\fill}}%
                >{$}#2<{$}@{\extracolsep{0pt}}%
                >{$\ }#3<{\ $}@{\extracolsep{0pt}}%
                >{$}#4<{$}@{\extracolsep{0pt}}%
                >{$\null}#5<{$}%
                @{\extracolsep{\fill}}>{\debutcom}r<{\fincomd}@{}}%
}{%
        \end{tabular*}%
        \vskip\calculskip
}

\newenvironment{calculs:latotale}[6][2]{%
        \calculvskip{#1}
        \noindent
        \begin{tabular*}{\linewidth}[t]%
                {@{}>{\debutcom}l<{\fincomg}@{\extracolsep{\fill}}%
                >{$\null}#2<{$}@{\extracolsep{0pt}}%
                >{$\null}#3<{$}@{\extracolsep{0pt}}%
                >{$\null}#4<{$}%
                @{\extracolsep{0pt}}>{$\null}#5<{$}%
                @{\extracolsep{0pt}}>{$\null}#6<{$}%
                @{\extracolsep{\fill}}>{\debutcom}r<{\fincomd}@{}}%
}{%
        \end{tabular*}%
        \vskip\calculskip
}

\newcommand{\f}[2]{{\ensuremath{%
        \mathchoice%
        {\dfrac{#1}{#2}}
    {\dfrac{#1}{#2}}
        {\frac{#1}{#2}}
        {\frac{#1}{#2}}
}}}

\renewcommand{\tilde}[1]{\ensuremath{\widetilde{#1}}}

\newcommand{\plus}{\mbox{\protect\raisebox{.2mm}{\tiny{\ensuremath{+}}}}}
\newcommand{\moins}{\mbox{\protect\raisebox{.2mm}{\tiny{\ensuremath{-}}}}}
\newcommand{\pinf}{\plus\ensuremath{\infty}}
\newcommand{\minf}{\moins\ensuremath{\infty}}

\newlength{\pmlength}
\newlength{\pmgraslength}
\newlength{\pmpetitlength}
\renewcommand{\leq}{\ensuremath{\leqslant}}
\renewcommand{\geq}{\ensuremath{\geqslant}}
\renewcommand{\epsilon}{\ensuremath{\varepsilon}}

\renewcommand{\setminus}{\smallsetminus}        % $\R \setminus \Q$
\newlength{\restsubwidth}
\newlength{\restsubheight}
\newlength{\restsubmoreheight}
\newcommand{\rest}[2]{%
        \settowidth{\restsubwidth}{\ensuremath{#2}}
        \settoheight{\restsubheight}{\ensuremath{{}_{#2}}}
        \ensuremath{{#1\hskip 0.5 pt}_{\vrule\kern2pt\parbox[b][%
        \the\restsubheight +
                \the\restsubmoreheight][b]{\the\restsubwidth}{%
                        \ensuremath{{}_{#2}}}}}
        }

\newcommand{\Sumtproto}[2]{%
        \ifthenelse{%
                \equal{#1}{}
        }{%
                \ifthenelse{%
                        \equal{#2}{}%
                }{%
                        \ensuremath{\sum}%
                }{%
                        \smash[b]{\ensuremath{\sum\limits_{#1}^{#2}}}%
                }
        }{%
                \ensuremath{\sum\limits}_{#1}^{#2}%
        }%
}

\newcommand{\SUM}[2]{\ensuremath{{\displaystyle\sum\limits_{#1}^{#2}}}}

\makeatletter
\newif\if@ListeStar
\global\@ListeStartrue

\newcommand{\liste}{%
        \@ifstar{\global\@ListeStartrue\@liste}%
                {\global\@ListeStarfalse\@liste}%
}

\newcommand{\@liste}[2][n]{%
        \if@ListeStar%
                \left({#2}_0,{#2}_1,\ldots,{#2}_{#1}\right)%
        \else%
                \left({#2}_1,{#2}_2,\ldots,{#2}_{#1}\right)%
        \fi\@ListeStarfalse%
}
\makeatother

\makeatletter
\newif\if@SuiteStar
\global\@SuiteStartrue

\newcommand{\suite}{%
        \@ifstar{\global\@SuiteStartrue\@suite}%
                {\global\@SuiteStarfalse\@suite}%
}

\newcommand{\@suite}[2][n]{%
        \if@SuiteStar%
                \left(#2_{#1}\right)_{#1\in\N^*}%
        \else%
                \left(#2_{#1}\right)_{#1\in\N}%
        \fi\@SuiteStarfalse%
}
\makeatother

\makeatletter
\newif\if@laststared
\global\@laststaredtrue
\newcommand{\mathBB}[1]{%
        \@ifstar%
        {\global\@laststaredtrue\m@thBB{#1}}%
        {\global\@laststaredfalse\m@thBB{#1}}%
}
\newcommand{\m@thBB}[1]{%
        \if@laststared{\ensuremath{\mathbb{#1}^{*}}\xspace}%
        \else{\ensuremath{\mathbb{#1}}\xspace}%
        \fi%
        \@laststaredfalse%
}
\makeatother

\newcommand{\N}{\ensuremath{\mathBB{N}}}                % Entiers
\newcommand{\R}{\ensuremath{\mathBB{R}}}                % Réels
\renewcommand{\P}{\ensuremath{\mathBB{P}}}

\newcommand{\G}{\ensuremath{\mathscr{G}}}       % La constante de gravitation
\newcommand{\TestGauche}[1]{\ifthenelse{\equal{#1}{}}{\minf}{#1}}
\newcommand{\TestDroite}[1]{\ifthenelse{\equal{#1}{}}{\pinf}{#1}}

\newcommand{\F}{\numero{F}}

        {\begin{pmatrix}}%
        {\end{pmatrix}}
        {\begin{vmatrix}}%
        {\end{vmatrix}}

\makeatletter
\renewcommand\section{\@startsection{section}{1}{\z@}%
	{1cm  \@plus -1ex \@minus -.2ex}%
	{2.3ex \@plus.2ex}%
	{\reset@font\normalsize\scshape\centering}}
\makeatother

\usepackage{palatino}  %police pour le mode texte
\usepackage[cm]{sfmath}  %police pour les maths (sans serif)

\usepackage{amsthm} %theoremes
\usepackage[bbgreekl]{mathbbol} %lettres grecques en gras

\usepackage{pst-node} %diagrammes
\usepackage{pstricks}

\usepackage{tweaklist} 
\renewcommand{\itemhook}{\setlength{\topsep}{1pt}%
\setlength{\itemsep}{0pt}}

\newtheorem{prop}{Proposition}[section]  
\newtheorem{de}[prop]{Definition}
\newtheorem{thm}[prop]{Theorem}
\newtheorem{lem}[prop]{Lemma}  
\newtheorem{cor}[prop]{Corollary}
\newtheorem*{thm2}{Theorem}

\newenvironment{rmk}{
	\refstepcounter{prop}%
	\noindent \textbf{Remark \theprop.}}{}
\newenvironment{example}{
	\refstepcounter{prop}%
	\noindent \textbf{Example \theprop.}
}{}
\newenvironment{notations}{
	\refstepcounter{prop}%
	\noindent \textbf{Notations \theprop.}
}{}
\newcommand{\leftcentersn}[4][2]{%
        \refstepcounter{prop}%
        \leftcentersright[#1]{#2}{#3}{%
                \textbf{(\theprop)}\label{#4}%
        }%
}
\newcommand{\bb}{Bialynicki-Birula}
\renewcommand{\G}{\mathbf{G}}
\newcommand{\B}{\mathbf{B}}
\newcommand{\U}{\mathbf{U}}
\newcommand{\V}{\mathbf{V}}
\newcommand{\T}{\mathbf{T}}
\renewcommand{\P}{\mathbf{P}}

\newcommand{\sRw}{{}^*\mathcal{R}_{\dot w}}
\renewcommand{\F}{\mathbb{F}}
\newcommand{\Y}{\mathrm{Y}}
\newcommand{\X}{\mathrm{X}}
\newcommand{\Yxw}{\mathrm{Y}_x(\dot w)}
\newcommand{\Yo}{\mathrm{Y}_{w_0}(\dot w)}
\newcommand{\Xo}{\mathrm{X}_{w_0}(w)}
\newcommand{\Yt}{\tilde{\mathrm{Y}}_x(\dot w)}
\renewcommand{\R}{\mathrm{R}\Gamma_c}
\newcommand{\Yg}{\mathrm{Y}_\gamma^\circ}
\newcommand{\ub}{U \backslash}
\newcommand{\db}{D(\U)^F \backslash}
\newcommand{\ol}{\mathop{\otimes}\limits^\mathrm{L}}
\renewcommand{\H}{\mathrm{H}_c}

\begin{document} 
\title{Deligne-Lusztig restriction \\ of a Gelfand-Graev module}
\author{Olivier Dudas\footnote{Laboratoire de Math\'ematiques, Universit\'e de Franche Comt\'e.}\footnote{
The author is partly supported by the ANR, Project No JC07-192339.}}

\maketitle

\begin{abstract}  Using Deodhar's decomposition of a double Schubert cell, we study the regular representations of finite groups of Lie type arising in the cohomology of Deligne-Lusztig varieties associated to tori. We deduce that the Deligne-Lusztig restriction of a Gelfand-Graev module is a shifted Gelfand-Graev module. 

\end{abstract}

\section*{Introduction}

Let $\G$ be a connected reductive algebraic group defined over an algebraic closure $\F$ of a finite field of characteristic $p$. Let $F$ be an isogeny of $\G$ such that some power is a Frobenius endomorphism. The finite group $G = \G^F$ of fixed points under $F$ is called a finite group of Lie type. We fix a maximal torus $\T$ contained in a Borel subgroup $\B$ with unipotent radical $\U$, all of which assumed to be $F$-stable. The corresponding Weyl group will be denoted by $W$.

\sk

In a attempt to have a complete understanding of the character theory of $G$, Deligne and Lusztig have introduced in \cite{DL} a family of biadjoint morphisms $\mathrm{R}_w$ and ${}^*\mathrm{R}_w$ indexed by $W$, leading to an outstanding theory of induction and restriction between $G$ and any of its maximal tori. Roughly speaking, they encode, into a virtual character, the different representations occurring in the cohomology of the corresponding Deligne-Lusztig variety. Unfortunately, the same construction does not give enough information in the modular setting, and one has to work at a higher level. More precisely, for a finite extension $\Lambda$ of the ring $\mathbb{Z}_\ell$ of $\ell$-adic integers, Bonnaf\'e and Rouquier have defined in \cite{BR1} the following functors:

\leftcenters{and}{$\begin{array}[b]{r@{\  \, : \,  \ }l} \mathcal{R}_{\dot w} & \mathcal{D}^b(\Lambda \T^{wF}$-$\mathrm{mod}) \longrightarrow \mathcal{D}^b(\Lambda G$-$\mathrm{mod}) \\[4pt] {}^*\mathcal{R}_{\dot w} & \mathcal{D}^b(\Lambda G$-$\mathrm{mod}) \longrightarrow \mathcal{D}^b(\Lambda \T^{wF}$-$\mathrm{mod}) \end{array}$}

\noindent between the derived categories of modules, which generalize the definition of Deligne-Lusztig induction and restriction. \sk

In this article we study the action of the restriction functor on a special class of representations: the Gelfand-Graev modules, which are projective modules  parametrized by the $G$-regular characters of $U$. More precisely, we prove in section 3 the following result:

\begin{thm2} Let $\psi : U \longrightarrow \Lambda^\times$ be a $G$-regular linear character, and denote by $\Gamma_\psi$ the associated Gelfand-Graev module of $G$. Then, for any $w$ in $W$, one has 

\centers{$\sRw \Gamma_{\psi} \, \simeq \, \Lambda \T^{wF} [-\ell(w)]$}

\noindent in the derived category $\mathcal{D}^b(\Lambda \T^{wF}$-$\mathrm{mod})$.

\end{thm2}

This result was already known for some specific elements of the Weyl group. In the case where $w$ is the trivial element, the Deligne-Lusztig functor ${}^*\mathcal{R}_{\dot w}$ comes from a functor defined at the level of module categories, and the result can be proved in a complete algebraic setting (see \cite[Proposition 8.1.6]{Car}). But more interesting is the case of a Coxeter element, studied by Bonnaf\'e and Rouquier in \cite{BR}. Their proof relies on the following geometric properties for the corresponding Deligne-Lusztig variety $\X(w)$ (see \cite{Lu}):

%\begin{itemize}
%
%\item $\X(w)$ is contained in the maximal Schubert cell $\B w_0  \B /\B$;
%
%\item the quotient variety $U \backslash \X(w)$ is a product of $\G_m$'s.
%
%\end{itemize}

\leftcentersright[1]{$\begin{array}{l} \quad \bullet \ \, \X(w) \text{ is contained in the maximal Schubert cell }\B w_0  \B /\B ; \\[4pt] \quad \bullet \ \, \text{the quotient variety } U \backslash \X(w) \text{ is a product of } \G_m\text{'s.} \end{array}$}{}{$\mathbf{(\star)}$}

\noindent Obviously, one cannot expect these properties to hold for any element $w$ of the Weyl group (for instance, the variety $\X(1)$ is a finite set of points whose intersection with any $F$-stable Schubert cell is non-trivial). However, it turns out that for the specific class of representations we are looking at, we can restrict our study to a smaller variety which will be somehow a good substitute for $\X(w)$.

%This subvariety is obtained as the maximal piece of a partition of $\X(w)$ coming from a general decomposition of double Schubert cells into locally closed smooth subvarieties. 

%More precisely, we show  in section 3 that there exists a partition of $\X(w)$ such that the only piece of this partition carrying regular characters in its cohomology satisfies the previous properties. The crucial ingredient of   
\sk
%\mk
 Let us give some consequences of this theorem, which are already known but can be deduced in an elementary way from our result.  From the quasi-isomorphism one can first obtain a  canonical algebra homomorphism from the endomorphism algebra of a Gelfand-Graev module to the algebra $\Lambda \T^{wF}$. Tensoring by the fraction field $K$ of $\Lambda$, it can be shown that we obtain the Curtis homomorphism ${}_K \mathrm{Cur}_w : \mathrm{End}_{KG}(K\Gamma_\psi) \longrightarrow K\T^{wF}$, thus giving a modular and conceptual version of this morphism (see \cite[Theorem 2.7]{BK}). 

%\sk
The character-theoretic version of the theorem is obtained in a drastic way, by tensoring the quasi-isomorphism by $K$ and by looking at the induced equality in the Grothendieck group of the category of $\Lambda \T^{wF}$-modules. Applying the Alvis-Curtis duality gives then a new method for computing the values of the Green functions at a regular unipotent element (see \cite[Theorem 9.16]{DL}). This is the key step for showing that a Gelfand-Graev character has a unique irreducible component in each rational series $\mathcal{E}(G,(s)_{{\G^*}^{F^*}})$.

%\sk 
Beyond these applications, our approach aims at understanding each of the cohomology groups of the Deligne-Lusztig varieties, leading to concentration and disjointness properties, in the spirit of Brou\'e's conjectures. For example, by truncating by unipotent characters, one can deduce that the Steinberg character is concentrated in the half cohomology group. This result was already proved in \cite[Proposition 3.3.15]{DMR}, by a completely different method, since their proof relies on the computation  of eigenvalues of Frobenius.  By refining our method, we should be able to deal with some other unipotent characters and enlarge the scope of our result.

\sk
%\mk
 This paper is divided into three parts: in the first section, we introduce the basic notations about the modular representation theory of finite groups of Lie type. Then, we focus on an extremely rich decomposition for double Schubert cells, introduced by Deodhar in \cite{Deo}. To this end, we shall use the point of view of \cite{Mo} and use the \bb\ decomposition, since it is particularly  adapted to our case. This is the crucial ingredient for proving the main theorem. Indeed, we show in the last section that the maximal piece of the induced decomposition on $\X(w)$ satisfies the properties $\mathbf{(\star)}$, and that it is the only one carrying regular characters in its cohomology.

%We show nevertheless that there exists a partition of $\X(w)$ such that the only piece of this partition carrying regular characters in its cohomology satisfies the previous properties.   

%\sk After recalling , we expose

%However, we show that there exists a partition of $\X(w)$ coming from Deohdar's decomposition of a double Schubert cell, such that the only piece of this partition carrying regular characters in its cohomology satisfies the previous properties.

\mk

\section*{Acknowledgments}

Part of this work was carried out while I was a Program Associate at the Mathematical Sciences Research Institute in Berkeley. I wish to thank the Institution for their support and hospitality and especially Arun Ram, who really helped me and the other graduate students to make the most of our stay. I am deeply indebted to C\'edric Bonnaf\'e for his suggestions, guidance, and encouragement throughout the course of this work. I also thank Geordie Williamson for fruitful discussions. 
\bk \pagebreak \newpage

\section{Preliminaries\label{prel}} 

\noindent \textbf{1 - Cohomology of a quasi-projective variety}

\mk Let $\Lambda $ be a commutative ring and $H$ a finite group. We denote by 
 $\Lambda H$-$\mathrm{mod}$ the abelian category of finitely generated $\Lambda H$-modules, and by $\mathcal{D}^b(\Lambda H$-$\mathrm{mod})$ the derived category of the corresponding bounded complexes. From now on, we assume that $\Lambda $ is a finite extension of the ring $\mathbb{Z}_\ell$ of $\ell$-adic integers, for a prime $\ell$ different from $p$. To any quasi-projective variety $\X$ defined over $\F$ and acted on by $H$, one can associate a classical object in this category, namely the cohomology with compact support of $\X$, denoted by $\R(\X,\Lambda)$.   It is quasi-isomorphic to a bounded complex of modules which have finite rank over $\Lambda$. 

\sk We give here some  quasi-isomorphisms we shall use in section \ref{proof}. The reader will find references or proofs of these properties in  \cite[Section 3]{BR1} and \cite[Proposition 6.4]{DL} for the third assertion. The last one can be deduced from \cite[Expos\'e XVIII, 2.9]{SGA4}.

\begin{prop}\label{coho}Let $\X$ and $\Y$ be two quasi-projective varieties acted on by $H$. Then one has the following isomorphisms in the derived category $\mathcal{D}^b(\Lambda H$-$\mathrm{mod})$:

\begin{itemize} 

\item[$\mathrm{(i)}$] The K\"unneth formula: 

\vskip -8pt \centers[1]{$\R(\X \times \Y, \Lambda) \, \simeq \, \R(\X , \Lambda) \, \ol\, \R(\Y,\Lambda)$}

\vskip -8pt \noindent where $\ol$ denotes the left-derived functor of the tensor product over $\Lambda$.

\item[$\mathrm{(ii)}$]  The quotient variety $H \backslash \X$ exists and 

\vskip -7pt \centers[0]{$\R(H \backslash \X, \Lambda) \, \simeq \, \Lambda \, {\ol}_{\Lambda H} \, \R(\X,\Lambda)$.}

 \item[$\mathrm{(iii)}$]  If the action of $H$ on $\X$ is the restriction of an action of a connected group, then 

\vskip -8pt \centers{$ \R(\X,\Lambda) \, \simeq \, \Lambda \, {\ol}_{\Lambda H} \, \R(\X,\Lambda ).$}

 \item[$\mathrm{(iv)}$] Let $\pi : Y \longrightarrow X$ be an $H$-equivariant smooth morphism of finite type. If the fibers of $\pi$ are isomorphic to affine spaces of constant dimension $n$, then 
 
 \centers[1]{$ \R(\Y,\Lambda ) \, \simeq \, \R(\X,\Lambda )[-2n].$}

\end{itemize}
\end{prop}

\sk

If $N$ is a finite group acting on $\X$ on the right and on $\Y$ on the left, we can form the amalgamated product $\X\times_N \Y$, as the quotient of $\X \times \Y$ by the diagonal action of $N$. Assume that the actions of $H$ and $N$ commute. Then $\X\times_N \Y$ is an $H$-variety and we deduce from the above properties that

\vskip -8pt \leftcentersn{}{$ \R(\X\times_N \Y, \Lambda ) \, \simeq \, \R(\X , \Lambda ) \, {\ol}_{\Lambda N}\, \R(\Y,\Lambda )$}{amalg}

\noindent in the derived category $\mathcal{D}^b(\Lambda H$-$\mathrm{mod})$.

\bk

\noindent \textbf{2 - Algebraic groups}
 
\mk We keep the basic assumptions of the introduction: $\G$ is a connected reductive algebraic group, together with an isogeny $F$ such that some power is a Frobenius endomorphism. In other words, there exists a positive integer $\delta$ such that $F^\delta$ defines a split $\F_q$-structure on $\G$ for a certain power $q$ of the characteristic $p$. For any algebraic subgroup $\mathbf{H}$ of $\G$ $F$-stable, we will denote by $H$ the finite group of fixed points $\mathbf{H}^F$.  \sk

We fix a Borel subgroup $\B$ containing a maximal torus $\T$ of $\G$ such that both $\B$ and $\T$ are $F$-stable. They define a root sytem $\Phi$ with basis $\Delta$, and a set of positive (resp. negative) roots $\Phi^+$ (resp. $\Phi^-$). Note that the corresponding Weyl group $W$ is endowed with a action of $F$, compatible with the isomorphism $W \simeq N_\G(\T)/\T$. Therefore, the image by $F$ of a root is a positive multiple of some other root, which will be denoted by $\phi^{-1}(\alpha)$, defining thus a bijection $\phi : \Phi \longrightarrow \Phi$. Since $\B$ is also $F$-stable, this map preserves $\Delta$ and $\Phi^+$. We will also use the notation $[\Delta/\phi]$ for a set of representatives of the orbits of $\phi$ on $\Delta$. 

\sk 
Let $\U$ (resp. $\U^-$) be the unipotent radical of $\B$ (resp. the opposite Borel subgroup $\B^-$). For any root $\alpha$, we denote by $\U_\alpha$ the one-parameter subgroup and $u_\alpha : \F \longrightarrow \U_\alpha$ an isomorphism. Note that the groups $\U$ and $\U^-$ are $F$-stable whereas $\U_\alpha$ might not be. However, we may, and we will, choose the family $(u_\alpha)_{\alpha \in \Phi}$ such that the restriction to $\U_\alpha$ of the action of $F$ satisfies $ F(u_{\alpha}(\zeta))  = u_{\phi(\alpha)}( \zeta^{q_\alpha^\circ})$ where $q_\alpha^\circ$ is some power of $p$ defined by the relation $F(\phi(\alpha)) = q_\alpha^\circ \,\alpha$.  We define  $d_\alpha$ to be the length of the orbit of $\alpha$ under the action of $\phi$ and  we set $q_\alpha = q_\alpha^\circ q_{\phi(\alpha)}^\circ \cdots q_{\phi^{d_\alpha-1}(\alpha)}^\circ$. Then $\U_\alpha$ is stable by $F^{d_\alpha}$ and $\U_\alpha^{F^{d_\alpha}} \simeq \F_{q_\alpha}$. 

\sk

Let us consider the derived group $D(\U)$ of $\U$. For any total order on $\Phi^+$, the product map induces the following isomorphism of varieties:

\centers{$ D(\U) \, \simeq \hskip-3mm  \displaystyle \prod_{\alpha \in \Phi^+ \setminus \Delta} \hskip-2mm \U_\alpha$.}

\noindent Note that it is not a group isomorphism in general, since $D(\U)$ might not be abelian. However, the canonical map $\prod_{\alpha \in \Delta} \U_\alpha \longrightarrow \U / D(\U)$ induces an isomorphism of algebraic groups commuting with $F$. In order to give a description of the rational points of this group, we introduce the one-parameter group $\V_\alpha$ as the image of the following morphism:

\centers{$ v_\alpha : \zeta \in \F \longmapsto \displaystyle \prod_{i=0}^{d_\alpha-1} F^i(u_\alpha(\zeta)) \, = \,  \prod_{i=0}^{d_\alpha-1} u_{\phi^i(\alpha)} (\zeta^{q_\alpha^\circ \cdots q_{\phi^{i-1}(\alpha)}^\circ})  \ \in \, \prod_{\alpha \in \Delta} \U_\alpha$.}

\noindent Then $\V_\alpha$ is $F$-stable and we obtain a canonical isomorphism of algebraic groups:

\centers{$ U / D(\U)^F \, \simeq \, (\U/ D(\U))^F \, \simeq  \hskip-2mm \displaystyle \prod_{\alpha \in [\Delta/\phi]} \hskip-1.5mm V_\alpha
% \, \simeq \hskip-3mm \displaystyle \prod_{\alpha \in [\Delta/\phi]} \hskip-2mm \F_{q_\alpha}
$.} 

We give now a construction of the quotient $D(\U)^F \backslash \B$ which will be useful in section 3. Towards this aim, we define the $F$-group $\B_{\Delta}$ to be

\centers{$ \B_{\Delta} \, = \, \Big( \displaystyle \prod_{\alpha \in \Delta} \U_\alpha \Big) \rtimes \T \, \simeq \, D(\U) \backslash \B$}

\noindent so that the total order on $\Phi^+$ chosen to describe $D(\U)$ gives rise to an isomorphism of varieties $\B \simeq D(\U) \times \B_\Delta$ which commutes with $F$. If $b \in \B$, we denote by $(b_D,b_\Delta)$ its image under this isomorphism. Conversely, we can embed the variety $\B_\Delta$ into $\B$ by considering the unique section $\iota : \B_\Delta \longrightarrow \B$ such that $b = b_D \iota(b_\Delta)$. We should notice again that if $W$ is not abelian, $\iota$ is only a morphism of algebraic varieties. With these notations, we obtain the following realization of $D(\U)^F \backslash \B$:

\begin{prop} The map $\varphi :  b \in \B \longmapsto  (b^{-1}F(b),b_\Delta) \in \B \times \B_\Delta $ induces the following isomorphism of varieties:

\centers{ $D(\U)^F \backslash \B \, \simeq \, \big\{ (\bar b, h) \in \B \times \B_\Delta \ \big| \ \bar b_\Delta = \mathcal{L}_\Delta(h) \big\}$}

\noindent where $\mathcal{L}_\Delta$ denote the Lang map corresponding to the $F$-group structure on $\B_\Delta$. 

\end{prop}

\begin{proof} Let $b$ be an element of $\B$; by definition, we can decompose it into $b=b_D \iota(b_\Delta)$ and hence compute  $b^{-1}F(b) = \big(b_D^{-1} F(b_D)\big){}^{\iota(b_\Delta)} \ \iota(b_\Delta)^{-1} F(\iota(b_\Delta))$. As a consequence, we get

\centers{$ \big(b^{-1}F(b) \big)_\Delta  \, = \, \big(\iota(b_\Delta)^{-1} F(\iota(b_\Delta)) \big)_\Delta \, = \, \mathcal{L}_\Delta(b_\Delta).$}

\noindent Therefore, the image of $\varphi$ is the set of pairs $(\bar b, h) \in \B \times \B_\Delta$ such that $\bar b_\Delta = \mathcal{L}_\Delta(h)$. 
\sk

As another application of the computation of $b^{-1}F(b)$, one can readily check that the fiber at a point $\varphi(b)$ is exactly the set $D(\U)^F b$.

\sk

Finally, we prove that $\varphi : \B \longrightarrow \mathrm{Im}\, \varphi$ is \'etale, which will prove the assertion of the proposition. Since the maps $\B \longrightarrow B\backslash \B$ and $\mathrm{Im}\, \varphi \longrightarrow B\backslash \mathrm{Im}\, \varphi$ are \'etale, it is sufficient to show that the induced map  $\varphi' :  B \backslash \B \longrightarrow B\backslash \mathrm{Im}\, \varphi$ is an isomorphism. But by the first projection $B \backslash \mathrm{Im} \, \varphi \longrightarrow \B$, $\varphi'$ identifies with the canonical isomorphism $Bb \in B\backslash \B \longmapsto b^{-1}F(b) \in \B$.
\end{proof}

%With these notations, using the fact that $b^{-1}F(b) = {}^{\iota(b_\Delta)} \big(b_D^{-1} F(b_D)\big)\, \iota(b_\Delta)^{-1} F(\iota(b_\Delta))$, one can easily check that

%\centers{$ \big(b^{-1}F(b) \big)_\Delta  \, = \, \big(\iota(b_\Delta)^{-1} F(\iota(b_\Delta)) \big)_\Delta \, = \, \mathcal{L}_\Delta(b_\Delta)$}

\bk

\noindent \textbf{3 - Deligne-Lusztig varieties}

\mk Following \cite[Section 11.2]{BR1}, we fix a set of representatives $\{\dot w\}$ of $W$ in $N_\G(\T)$ and we define, for $w \in W$, the Deligne-Lusztig varieties $\X(w)$ and $\Y(\dot w)$ by:

\sk

\centers{$ \begin{psmatrix}[colsep=2mm,rowsep=10mm] \Y(\dot w) & = \, \big\{ g\U \in \G / \U \ \big| \ g^{-1}F(g) \in \U \dot w \U \big\} \\
					\X(w) & = \, \big\{ g\B \in \G / \B \ \big| \ g^{-1}F(g) \in \B w \B \big\} 
\psset{arrows=->>,nodesep=3pt} 
\everypsbox{\scriptstyle} 
\ncline{1,1}{2,1}<{\pi_w}>{/ \, \T^{wF}}		
\end{psmatrix}$}

\sk

\noindent where $\pi_w$ denotes the restriction to $\Y(\dot w)$ of the canonical projection $\G/\U \longrightarrow \G/\B$. They are both quasi-projective varieties endowed with a left action of $G$ by left multiplication. Furthermore, $\T^{wF}$ acts on the right of $\Y(\dot w)$ and $\pi_w$ is isomorphic to the corresponding quotient map, so that it induces a $G$-equivariant isomorphism of varieties $\Y(\dot w) / \T^{wF} \simeq \X(w)$. 

\sk

We introduce now  the general framework for the modular representation theory of $G$: we choose a prime number $\ell$ different from $p$ and we consider an $\ell$-modular system $(K,\Lambda,k)$ consisting on a finite extension $K$ of the field of $\ell$-adic numbers $\mathbb{Q}_\ell$, the normal closure $\Lambda$ of the ring of $\ell$-adic integers in $K$ and the residue field $k$ of the local ring $\Lambda$. We assume moreover that the field $K$ is big enough for $G$, so that it contains the $e$-th roots of unity, where $e$ is the exponent of $G$. In that case, the algebra $KG$ is split semi-simple.
\sk 

By considering the object $\R(\Y(\dot w), \Lambda)$ of the category $\mathcal{D}^b(\Lambda G$-$\mathrm{mod}$-$\Lambda \T^{wF})$, we define a pair of biadjoint functors between the derived categories $\mathcal{D}^b(\Lambda G$-$\mathrm{mod})$ and $\mathcal{D}^b(\Lambda \T^{wF}$-$\mathrm{mod})$, called the \textbf{Deligne-Lusztig induction and restriction functors}:

\vskip -4pt \leftcenters{and}{$\begin{array}[b]{r@{\, \ = \ \,}l}  \mathcal{R}_{\dot w} & \R(\Y(\dot w),\Lambda) \, {\ol}_{\T^{wF}} - \ : \,  \mathcal{D}^b(\Lambda \T^{wF}\mathrm{\text{-}mod}) \longrightarrow \mathcal{D}^b(\Lambda G\mathrm{\text{-}mod}) \\[6pt] 
{}^*\mathcal{R}_{\dot w} & \mathrm{R}\mathrm{Hom}_{\Lambda G}^\bullet(\R(\Y(\dot w),\Lambda), -) \ : \,  \mathcal{D}^b(\Lambda G\mathrm{\text{-}mod}) \longrightarrow \mathcal{D}^b(\Lambda \T^{wF}\mathrm{\text{-}mod}) \end{array}$}

\noindent where $\mathrm{Hom}_{\Lambda G}^\bullet$ denotes the classical bifunctor of the category of complexes of $\Lambda G$-modules. Tensoring by $K$, they induce adjoint morphisms $\mathrm{R}_w$ and ${}^*\mathrm{R}_w$ between the corresponding Grothendieck groups, known as the original induction and restriction defined in \cite{DL}. Note that when $w=1$, these functors are the reflect of the Harish-Chandra induction and restriction, which are more simply defined at the level of the modules categories.

\bk

\noindent \textbf{4 - Gelfand-Graev modules}

\mk In this subsection, we present the basic definitions and results concerning the Gelfand-Graev representations, before stating the main theorem. We discuss then the strategy suggested by Bonnaf\'e and Rouquier in \cite{BR}, coming from some geometric observations in the case of a Coxeter element.

\sk

Let $\psi \, : \, U \longrightarrow \Lambda^\times$ be a linear character of $U$. We assume that $\psi$ is trivial on $D(\U)^F$; it is not a strong condition since the equality $D(\U)^F = D(U)$ holds in most of the cases (the only exceptions for quasi-simple groups being groups of type $B_2$ or $F_4$ over $\F_2$ or groups of type $G_2$ over $\F_3$, see \cite[Lemma 7]{How}). For any $\alpha$ in $[\Delta/\phi]$, we denote by $\psi_\alpha$ the restriction of $\psi$ through $\F_{q_\alpha} \simeq V_\alpha \hookrightarrow U/D(\U)^F$. 

\begin{de} A linear character $\psi \, : \, U \longrightarrow \Lambda^\times$ is said to be \textbf{$G$-regular} if

\begin{itemize} 
\item $\psi$ is trivial on $D(\U)^F$;
\item $\psi_\alpha$ is non-trivial for any $\alpha$ in  $[\Delta/\phi]$.

\end{itemize}

\noindent In that case, we define the \textbf{Gelfand-Graev module} of $G$ associated to $\psi$ to be

\centers{$ \Gamma_{\psi}  \, = \, \mathrm{Ind}_U^G \, \Lambda_\psi$}

\noindent where $\Lambda_\psi$ denotes the $\Lambda U$-module of $\Lambda$-rank $1$  on which $U$ acts through $\psi$.\end{de}

%When the ambiant group is clear, we will use the notation $\Gamma_\psi$ instead of $\Gamma_{\psi,G}$.

%It is a projective and multiplicity free module (see \cite{St}). 

\begin{rmk}\ One can perform the same construction without assuming that $\psi$ is regular. However, Steinberg has shown in \cite{St} that in the case of a regular character, the module obtained is always multiplicity free. For example,
if $\G$ is a torus, then the unipotent group is trivial, and therefore there is only one Gelfand-Graev module, corresponding to the trivial character: the regular representation $\Lambda G$. \end{rmk}

\mk

The construction of the Gelfand-Graev modules is to some extent orthogonal from the Deligne-Lusztig functors. In this sense, we might try to understand the interactions between the two notions. The following theorem and main result of this article is a partial achievement in this direction; it asserts that the Deligne-Lusztig restriction of a Gelfand-Graev module is  a shifted Gelfand-Graev module. By the previous remark, it can be stated as follows:

\begin{thm}\label{gg}Let $\psi : U \longrightarrow \Lambda^\times$ be a $G$-regular linear character and $w$ be an element of $W$. Then

\centers{$\sRw \Gamma_{\psi} \, \simeq \, \Lambda \T^{wF} [-\ell(w)]$}

\noindent in the derived category $\mathcal{D}^b(\Lambda \T^{wF}$-$\mathrm{mod})$.

\end{thm}

Such a result was already known for some specific elements of the Weyl group, namely for the trivial element (see \cite[Proposition 8.1.6]{Car}) and for a Coxeter element (see \cite[Theorem 3.10]{BR}). 
As explained in the introduction, both of these cases suggested to Bonnaf\'e and Rouquier that a stronger result should hold.  Before recalling it, we need to introduce some more notations. For $x \in W$, we will denote by $\B x \cdot \B$ the unique $\B$-orbit of $\G / \B$ containing $x$, and will refer to it as the \textbf{Schubert cell} corresponding to $x$. Following \cite{BR}, we now define the pieces of the Deligne-Lusztig varieties:

\leftcenters{and}{$ \begin{array}[b]{rl} \Y_x(\dot w) & \hskip -1.7mm = \, \big\{ g\U \in \B x \cdot \U \ \big | \ g^{-1}F(g) \in \U \dot w \U \big\} \\[6pt]
%& \hskip-1.7mm \simeq \, \big\{ b \in \B\cap {}^x \B^- \ \big| \ b^{-1}F(b) \in {}^x(\U \dot w \U) \big\}
%\end{array}$}
%\noindent where the last isomorphism is induced by $b \in \B \cap {}^x \B^- \longmapsto bx \U \in \B x \cdot \U$. The same construction holds for the variety $\X(w)$ :
%\centers{$ \begin{array}{rl} 
\X_x(w)& \hskip -1.7mm = \, \big\{ g\B \in \B x \cdot \B \ \big | \ g^{-1}F(g) \in \B w \B \big\}.
%& \hskip-1.7mm \simeq \, \big\{ u \in \U\cap {}^x \U^- \ \big| \ u^{-1}F(u) \in {}^x(\B w \B) \big\}
\end{array}$}

\noindent By the Bruhat decomposition and its analog for $\G/ \U$, one can check that this gives a filterable decomposition of the varieties $\Y(\dot w)$ and $\X(w)$. Finally, if $\psi : U \longrightarrow \Lambda^\times$ is a linear character of $U$, we denote by $e_\psi$ the corresponding idempotent, defined by

\centers{$e_\psi \, = \, \f{\psi(1)}{| U |}\ \, \SUM{u\in U}{} \ \, \psi(u^{-1}) u.$}

\noindent Since $U$ is a $p$-group, its order is invertible in $\Lambda$ and $e_\psi$ is an central element of $\Lambda U$. 

\sk
With these notations, the result conjectured by Bonnaf\'e and Rouquier in \cite[Conjecture 2.7]{BR} can be stated as follows:

\begin{thm}\label{br}Let $\psi : U \longrightarrow \Lambda^\times$ be a $G$-regular linear character and $w$ be an element of $W$. Then

\centers{$e_\psi \R(\Yxw, \Lambda) \, \simeq \, \left\{ \hskip -1.3mm \begin{array}{l} \Lambda \T^{wF} [-\ell(w)] \ \ \text{if } x=w_0 \\[4pt] 0 \ \ \text{otherwise} \end{array} \right.$}

\noindent in the derived category $\mathcal{D}^b(\mathrm{mod}$-$\Lambda \T^{wF})$.

\end{thm}

 Note that the family of varieties $(\Y_x(\dot w))_{x \in W}$ is a filterable decomposition of $\Y(\dot w)$ with maximal element $\Y_{w_0}(\dot w)$ and that
%\centers
{$ \sRw \Gamma_\psi \, \simeq\, \mathrm{R} \mathrm{Hom}_{\Lambda U}^\bullet (\R(\Y(\dot w),\Lambda), \Lambda_\psi) \, \simeq \, \big(e_\psi \R (\Y( \dot w), \Lambda) \big)^\vee$} in the derived category $\mathcal{D}^b(\Lambda \T^{wF}$-$\mathrm{mod})$, so that theorem \ref{gg} follows indeed from this result.

\sk

To conclude this section, we move back attention to the family of varieties $(\Y_x(\dot w))_{x \in W}$  and $(\X_x(w))_{x \in W}$. By the isomorphism $\B \cap {}^x \B^- \simeq \B x\cdot \U$ which sends $b$ to the coset $bx\U$, and its analog for $\G/\B$, we obtain a new description of these varieties:

\leftcenters{and}{$ \begin{array}[b]{rl} \Y_x(\dot w) & \hskip -1.7mm \simeq \, \big\{ b  \in \B \cap {}^x \B^- \ \big | \ b^{-1}F(b) \in x(\U \dot w \U)F(x)^{-1} \big\} \\[6pt]
\X_x(w)& \hskip -1.7mm \simeq \, \big\{ u \in \U \cap {}^x\U^- \ \big | \ u^{-1}F(u) \in x(\B w \B)F(x)^{-1} \big\}
\end{array}$}

\noindent the right action of $\T^{wF}$ on $\Y_x(\dot w)$ being now twisted by $x$. With this description, the restriction of $\pi_w$ sends an element $b$ of $\Y_x(\dot w)$ to its  left projection on $\U$, that is, according to the decomposition $\B = \U \rtimes \T$.

\section{Deodhar's decomposition \label{ddec}}

We recall in this section the principal result of \cite{Deo}, using a different approach due to Morel (see \cite[Section 3]{Mo}) which relies on a general decomposition theorem, namely the \bb\ decomposition, applied to Bott-Samelson varieties.  We also give an elementary result on the filtration property of this partition. Both are the fundamental tools  we will be using in the next section.

\bk

\noindent \textbf{1 - Decomposition of a double Schubert cell}

\mk

Let $w \in W$ be an element of the Weyl group of $\G$.  The Schubert variety $\mathrm{S}_w$ associated to $w$ is the closure in $\G/\B$ of the Schubert cell $\B w\cdot \B$. This variety is not smooth in general, but Demazure has constructed in \cite{Dem} a resolution of the singularities, called the Bott-Samelson resolution, which is a projective smooth variety over $\mathrm{S}_w$. The construction is as follows: we fix a reduced expression $w=s_1 \cdots s_r$ of $w$ and we define the Bott-Samelson variety to be

\centers{$\mathbf{BS} = \P_{s_1} \times_\B \cdots \times_\B \P_{s_r} /\B$}

\noindent where $\P_{s_i}= \B \cup \B s_i\B$ is the standard parabolic subgroup corresponding to the simple reflection $s_i$. The homomorphism $\pi : \mathbf{BS} \longrightarrow \mathrm{S}_w$ which sends the class  $[p_1,\ldots,p_r]$ in $\mathbf{BS}$ of an element  $(p_1,\ldots,p_r) \in  \P_{s_1} \times \cdots \times \P_{s_r} $ to the class of the product $p_1 \cdots p_r$ in $\G/\B$ is called the \textbf{Bott-Samelson resolution}. It is a proper surjective morphism of varieties and it induces an isomorphism between $\pi^{-1}(\B w\cdot \B)$ and $\B w\cdot \B$.

\sk

Now the torus $\T$ acts naturally on $\mathbf{BS}$ by left multiplication on the first component, or equivalently by conjugation on each component,  so that $\pi$ becomes a $\T$-equivariant morphism. There are finitely many fixed points for this action, represented by the classes of the elements of $\Gamma = \{1,s_1\} \times \cdots \times \{1,s_r\}$ in $\mathbf{BS}$; such an element will be called a \textbf{subexpression} of~$w$. 

\sk

For a subexpression $\gamma = (\gamma_1,\ldots, \gamma_r) \in \Gamma$ of $w$, we denote by $\gamma^i = \gamma_1 \cdots \gamma_i$ the $i$-th partial subword and we define the two following sets:

\leftcenters{and}{ $\begin{array}[b]{r@{\ \, = \ \, }l} 
I(\gamma) & \big\{i \in \{1, \ldots, r \} \ | \ \gamma_i = s_i \big\} \\[4pt]
J(\gamma) & \big\{i \in \{1,\ldots, r \} \ | \ \gamma^i s_i < \gamma^i \big \}.\end{array}$}

\noindent With these notations,  Deodhar's decomposition theorem (see \cite[Theorem 1.1 and Corollary 1.2]{Deo}) can be stated as follows:

\begin{thm} [Deodhar, 84]\label{deo}There exists a set $\{D_\gamma\}_{\gamma \in \Gamma}$ of disjoint smooth locally closed subvarieties of $\B w\cdot \B$ such that: \begin{itemize}

\item[$\mathrm{(i)}$]  $D_\gamma$ is non-empty if and only if $J(\gamma) \subset I(\gamma)$;

\item[$\mathrm{(ii)}$]  if  $D_\gamma$ is non-empty, then it is isomorphic to $(\G_a)^{|I(\gamma)|- |J(\gamma)|} \times (\G_m)^{r-|I(\gamma)|}$ as a variety;

\item[$\mathrm{(iii)}$] for all $v \in W$, the double Schubert cell has the following decomposition:

\centers{$ \displaystyle \B w \cdot \B \cap \B^- v \cdot \B = \coprod_{\gamma \in \Gamma_v} D_\gamma$}

\noindent where $\Gamma_v$ is the subset of $\Gamma$ consisting of all subexpressions $\gamma$ such that $\gamma^r = v$.
\end{itemize} 

\end{thm}

\mk

\begin{rmk}\ In the first assertion, the condition for a cell $D_\gamma$ to be non-empty, that is $J(\gamma) \subset I(\gamma)$,  can be replaced by:

\centersright{$\forall \, i=2, \ldots, r \qquad  \gamma^{i-1} s_i < \gamma^{i-1} \ \Longrightarrow \ \gamma_i = s_i$.}{}

\noindent A subexpression $\gamma \in \Gamma$ which satisfies this condition is called a \textbf{distinguished subexpression}. For example, if $\G = \mathrm{GL}_3(\F)$ and $w=w_0 = s t s$, then there are seven distinguished subexpressions, the only one being not distinguished is $(s,1,1)$.
\end{rmk}

\begin{proof}[Sketch of proof:] the Bott-Samelson variety is a smooth projective variety endowed with an action of the torus $\T$. Let us consider the restriction of this action to $\mathbf{G}_m$ through a strictly dominant cocharacter $\chi : \mathbf{G}_m \longrightarrow \T$. Since this action has a finite number of fixed points, namely the elements of $\Gamma$, there exists a Bialynicki-Birula decomposition of the variety $\mathbf{BS}$ into a disjoint union of affine spaces indexed by $\Gamma$ (see \cite[Theorem 4.3]{BB}):

\centers{$ \displaystyle \mathbf{BS} = \coprod_{\gamma \in \Gamma} C^\gamma$}

\noindent In \cite{Ha}, H\"arterich has explicitly computed the cells $C^\gamma$. To describe this computation, we need some more notations:  the simple roots corresponding to the simple reflections of the reduced expression $w = s_1 \cdots s_r$ will be denoted by $\beta_1,\ldots,\beta_r$ and we set $\tilde \beta_i = \gamma^i(-\beta_i)$. Note that with these notations, one has

\centers{$ J(\gamma) = \big\{ i \in \{1, \ldots, r \} \ | \ \tilde \beta_i \in \Phi^+ \big\}.$}

\sk

\noindent Let us consider the open immersion $a_\gamma : \mathbb{A}_r \longrightarrow \mathbf{BS}$ defined by 

\centers{ $ a_\gamma(x_1,\ldots,x_r) = [u_{\gamma_1(-\beta_1)}(x_1) \gamma_1, \ldots, u_{\gamma_r(-\beta_r)}(x_r) \gamma_r ] .$}

\noindent Then one can easily check that $\pi^{-1}(\B w \cdot \B) = \mathrm{Im}(a_{(s_1,\cdots,s_r)})$. Moreover, H\"arterich's computations (see \cite[Section 1]{Ha}) show that for any subexpression $\gamma \in \Gamma$, one has: 

\centers{$ C^\gamma \, = \, a_\gamma\big(\{ (x_1, \ldots,x_r) \in \mathbb{A}_r \ | \ x_i = 0 \ \text{ if }\ i \in J(\gamma) \} \big).$}

\noindent Taking the trace of this decomposition with $\pi^{-1}(\B w\cdot \B)$, one obtains a decomposition of the variety $\pi^{-1}(\B w\cdot \B)$. Furthermore, the restriction of $\pi$ to this variety induces an isomorphism with $\B w\cdot \B$, and thus gives a partition of $\B w\cdot \B$ into disjoint cells:

\centers{$ \pi^{-1}(\B w\cdot \B) \, = \, \displaystyle \coprod_{\gamma \in \Gamma}  \pi^{-1}(\B w\cdot \B) \cap C^\gamma \, \simeq \, \coprod_{\gamma \in \Gamma}  \B w\cdot \B \cap \pi(C^\gamma) \, = \, \B w\cdot \B.$}

\noindent If we define $D_\gamma$ to be the intersection $\B w\cdot \B \cap \pi(C^\gamma)$, then it is explicitly given by:

\centers{$ D_\gamma \simeq \pi^{-1}(D_\gamma) \, = 
\, a_\gamma\big(\{ (x_1, \ldots,x_r) \in \mathbb{A}_r \ | \ x_i = 0 \ \text{ if }\ i \in J(\gamma) \ \text{ and } \ x_i \neq 0 \ \text{ if } \ i \notin I(\gamma) \} \big).$}

\noindent This description, together with the inclusion $\pi(C^\gamma) \subset \B^- \gamma^r \cdot \B$, proves the three assertions of the theorem.\end{proof}

\begin{example}\ \label{ex}In the case where $\G = \mathrm{GL}_3(\F)$, and $w = w_0 = sts$, one can easily describe the double Schubert cell $\B w\cdot \B \cap \B^-\cdot \B$. It is isomorphic to $\B w\B \cap \U^-$ by the map $u \mapsto u\B$. Besides, by Gauss reduction, the set $\B w \B w^{-1} = \B\B^-$ consists of all matrices whose principal minors are non-zero. Hence, 

\centers{$ \B w\B \cap \U^- \, = \, \left\{ \left( \begin{array}{ccc} 1 & 0 & 0 \\ b& 1 & 0 \\ c & a & 1 \end{array} \right) \Big| \ c \neq 0 \ \text{ and } \ ab-c \neq 0 \right\}.$}

\noindent Considering the alternative $a = 0$ or $a \neq 0$, one has $\B w\B \cap \U^- \simeq (\G_m)^3 \cup {\G_a} \times {\G_m}$, which is exactly the decomposition given by the two distinguished expressions $(1,1,1)$ and $(s,1,s)$. \end{example}

\mk

\begin{notations}\ \label{not}For a subexpression $\gamma \in \Gamma$, we define the sequence

\centers{$\begin{array}{r@{\ \, = \, \ }l} \Phi(\gamma) & \big(\gamma^i(-\alpha_i) \ \big| \  i=1,\ldots, \ell \ \text{and} \ \gamma^i(\alpha_i) > 0\big)  \\[4pt] & \big( \tilde \beta_i \ \big| \  i=1,\ldots, \ell \ \text{and} \ \gamma^i(\alpha_i) > 0\big). \end{array}$}

\noindent Using H\"arterich's computation for the cell $C^\gamma$ and the definition of $\pi$, one can see that each element of $\pi(C^\gamma) \subset \B^- \gamma^r \cdot \B$ has a representative in $\U^-$ of the form

\centers{$\displaystyle \prod_{\beta \in \Phi(\gamma)} u_\beta (x_\beta) \qquad \text{with each} \ \ x_\beta \in {\G_a}$,}

\noindent the product being taken with respect to the order on $\Phi(\gamma)$. At the level of $D_\gamma$, some of the variables $x_\beta$ must be non-zero (those corresponding to $\tilde \beta_i$ with $\gamma_i = 1$) but the expression becomes unique.  Note that this set of representatives is not contained in the variety $\B w \B \cap (\U^- \cap {}^v  \U^-)$ in general. However, this is the case for $v =1$, and we obtain in this way a parametrization of the variety $\B w \B \cap \U^-$. \end{notations}

\bk

\noindent \textbf{2 - Filtration property}

\mk

We keep the previous notations: $v$ and $w$ are two elements of the Weyl group such that $v \leq w$ and $w=s_1 \cdots s_r$ is a reduced expression of $w$. Deodhar has shown (see \cite[Proposition 5.3.(iv)]{Deo})  that among all the distinguished subexpressions of $w$ ending by $v$,  there exists one and only one for which the sets $J$ and $I$ are equal. For example, if $v=1$, it is clear that this subexpression is $(1,\ldots,1)$. By theorem \ref{deo}.(ii) the corresponding Deodhar cell has the following properties: 

\begin{itemize}
\item it is the unique maximal cell in $\B w \cdot \B \cap \B^- v \cdot \B$, and it is of dimension $\ell(w)- \ell(v)$;
\item it is a product of $\mathbf{G}_m$'s;
\item it is dense in $\B w \cdot \B \cap \B^- v \cdot \B$ (since this variety is irreducible by \cite{Ri}).
\end{itemize}

\noindent   In particular, the border of the maximal cell is a union of cells of lower dimensions. Unfortunately, this is not always true for the other cells (see \cite{Du}), and the decomposition is not a stratification in general. However, it is possible to describe some relations between the different closures, showing that it is  at least filterable. This is a general property for projective smooth varieties (see \cite{BB2}) but we shall give here a simple method for constructing the filtration. In this aim, we can embed 
%the Bott-Samelson variety 
$\mathbf{BS}$ into a product of flag varieties as follows: we define the morphism $\iota : \mathbf{BS} \longrightarrow (\G/\B)^r$ by

\centers{$  \iota([p_1,p_2, \ldots,p_r]) =(p_1\B,p_1p_2\B, \ldots, p_1 p_2 \cdots p_r \B).$ }

\noindent Note that $\pi$ is the last component of this morphism. Let $\gamma \in \Gamma$ be a subexpression of $w$. As a direct consequence of the construction of $C^\gamma$, one has 

\centers{$\iota(C^\gamma) \ \subset \ \displaystyle \prod_{i=1}^r \B^-\gamma^i \cdot \B. $}

\noindent Since $\mathbf{BS}$ is projective, $\iota$ is a closed morphism, and hence it sends the closure of a cell $C^\gamma$ in $\mathbf{BS}$ to the closure of $\iota(C^\gamma)$. Therefore, it is natural to consider a partial order on the set $\Gamma$ coming from to the Bruhat order on $W$ since it describes the closure relation for Schubert cells. For $\delta \in \Gamma$, we define

\centers{$ \delta \preceq \gamma \ \ \iff \ \ \gamma^i \leq \delta^i\ $ for all $\ i = 1, \dots, r.$}

\leftcentersn{Then, by construction:}{ $\overline{C^\gamma} \ \subset \ \displaystyle \bigcup_{\delta \preceq \gamma} C^\delta$ \quad and \quad $\overline{D_\gamma} \ \subset \ \displaystyle \bigcup_{\delta \preceq \gamma} D_\delta$}{clos}

\noindent where $\overline{D_\gamma}$ denotes the closure of $D_\gamma$ in the Schubert cell $\B w\cdot \B$. Furthermore, considering subexpressions in $\Gamma_v$ only would lead to a similar description for the closure in $\B w \cdot \B \cap \B^- v \cdot \B$. Therefore,

\begin{lem}\label{filt}Let $w,v \in W$ such that $v \leq w$ and $w=s_1 \cdots s_r$ be a reduced expression of $w$. Then there exists a numbering of $\Gamma_v = \{\gamma_0, \gamma_1, \ldots, \gamma_n\}$ and a sequence $(F_i)_{i=0,\ldots,n}$ of closed subvarieties of $\B w \cdot \B \cap \B^- v \cdot \B$ such that:

\begin{itemize}
\item $\gamma_0$ corresponds to the maximal cell;
\item $F_n = D_{\gamma_n}$;
\item $\forall \, i =0,\ldots,n-1$, \ \ $F_{i} \setminus F_{i+1} \, = \, D_{\gamma_i}$.
\end{itemize}
\end{lem}

\begin{proof} We can choose any numbering which refines the order $\preceq$: the $(F_i)$ are then determined by the relation they must verify, and they are closed as a consequence of formula \ref{clos}.  
\end{proof}

\section{Proof of the main theorem\label{proof}}

In this section, we present a proof of theorem \ref{br}, which, as explained in section \ref{prel}.4, implies the main result of this article.  It is divided into two parts: in the first subsection, we deal with the variety $\Yxw$ for an element $x \in W$ which is assumed to be different from $w_0$. We show that there cannot be any regular character in the cohomology of this variety. The key point is to consider a specific quotient of this variety for which one can extend the action of at least one of the finite groups $V_\alpha$ up to the corresponding one-parameter subgroup $\V_\alpha$. The second subsection is devoted to the remaining piece, the variety $\Y_{w_0}(\dot  w)$. The crucial ingredient for studying its cohomology is the Deodhar decomposition, which has been recalled in the previous section. We shall lift this decomposition to obtain a nice partition of $\Y_{w_0}(\dot  w)$ and then compute the cohomology of the pieces using the results of \cite{BR}.

\bk

\noindent \textbf{1 - The variety $\Yxw$ for $x \neq w_0$}\mk

Throughout this section, $x$ will denote any element of the Weyl group $W$ different from $w_0$. In order to study the cohomology groups of the Deligne-Lusztig variety $\Yxw$, we define 

\centers{$ \Yt \, = \, \big\{b \in \B \ \big| \ b^{-1}F(b) \in x(\U \dot w \U) F(x)^{-1} \big\}.$}

\noindent It is a $B$-variety, endowed with a right action of $\T^{wF}$ obtained by right multiplication by the conjugate $x \T^{wF} x^{-1}$. Besides, it is related to 
%the Deligne-Lusztig variety 
$\Yxw$ via the map $\pi : b \in \Yt \longmapsto bx\U \in \Yxw$. One can readily check that this morphism has the following properties:

\begin{itemize}
\item it is a surjective smooth $B \times (\T^{wF})^{\mathrm{op}}$-equivariant morphism of varieties;
\item all the fibers of $\pi$ are isomorphic to an affine space of dimension $\ell(w_0) - \ell(x)$.
\end{itemize}

\noindent It is therefore sufficient to study the cohomology of $\Yt$. More precisely, one has by proposition \ref{coho}.(iv)

\leftcentersn{}{$ \R(\Yxw,\Lambda) \, \simeq \, \R(\Yt, \Lambda)[2(\ell(w_0)-\ell(w))]$}{tild}

\noindent in the derived category $\mathcal{D}^b(\Lambda B$-$\mathrm{mod}$-$\Lambda \T^{wF})$. We now prove the following result:

\begin{prop}\label{ext}Let $\alpha$ be a positive simple root such that $x^{-1}(\alpha) >0$. 
%We define the one-parameter subgroup $V_\alpha$ as the image of the morphism $\lambda \in \F \longmapsto u_\alpha(\lambda) F(u_\alpha(\lambda)) \cdots F^{d_\alpha-1}(u_\alpha(\lambda))$. 
Then the action of $V_\alpha$ on the quotient $D(\U)^F \backslash \Yt$ extends to an action of  $\V_\alpha$.
\end{prop}

\begin{proof} 
%As in the previous section, we must first realize the quotient $D(\U)^F \backslash \B$ by extracting $D(\U)$ from $\B$. In this aim, we define the $F$-group $\B_{\Delta}$ to be

%\centers{$ \B_{\Delta} \, = \, \Big( \displaystyle \prod_{\alpha \in \Delta} \U_\alpha \Big) \rtimes \T \, \simeq \, D(\U) \backslash \B$}

%\noindent and we fix an isomorphism of $F$-varieties $\B \simeq D(\U) \times \B_\Delta$ (which amounts to fix an order on $\Delta$). If $b \in \B$, we denote by $(b_D,b_\Delta)$ its image under this isomorphism. %Conversely, we can embed the variety $\B_\Delta$ into $\B$ by considering the unique section $\iota : \B_\Delta \longrightarrow \B$ such that $b = b_D \iota(b_\Delta)$. Note that if $W$ is not abelian, $\iota$ is not a group homomorphism. 
%Then, by considering the image of the morphism $ b \in \B \longmapsto  (b^{-1}F(b),b_\Delta) \in \B \times \B_\Delta $, one can easily check (see ?? for details) that 

%\noindent Using the fact that $b^{-1}F(b) = {}^{\iota(b_\Delta)} \big(b_D^{-1} F(b_D)\big)\, \iota(b_\Delta)^{-1} F(\iota(b_\Delta))$, one can easily check that

%\centers{$ \big(b^{-1}F(b) \big)_\Delta  \, = \, \big(\iota(b_\Delta)^{-1} F(\iota(b_\Delta)) \big)_\Delta \, = \, \mathcal{L}_\Delta(b_\Delta)$}

%\noindent where $\mathcal{L}_\Delta$ is the Lang map corresponding to the $F$-group structure on $\B_\Delta$. 

Considering the restriction to $\Yt$ of the map $ b \in \B \longmapsto  (b^{-1}F(b),b_\Delta) \in \B \times \B_\Delta $ studied in section \ref{prel}.2 leads to the following parametrization of the quotient:

\centers{$D(\U)^F \backslash \Yt \simeq \big\{ (\bar b,h) \in \B \times \B_\Delta  \ \big| \ \bar b _\Delta = \mathcal{L}_\Delta(h) \ \mathrm{and} \ \bar b \in x(\U \dot w \U) F(x)^{-1}\big\}.$}

\noindent We should notice that with this description, $U$ (and then $D(\U)^F \backslash U$) acts only on the second coordinate. More precisely, any element $u$ of $U$ acts by $ u \cdot (\bar b, h) \, = \, (\bar b, u_\Delta h)$. 

\sk

Through the quotient map $\B_\Delta \longrightarrow \T$, one can extend trivially any character of $\T$ to the group $\B_\Delta$. For the simple root $\alpha$, we will denote by $\tilde \alpha : \B_\Delta \longrightarrow \mathbf{G}_m$ the corresponding extension and we define an action of $\V_\alpha$ on the quotient variety $D(\U)^F \backslash \Yt$ by:

\centers{$ \forall \, \zeta \in \F \qquad v_\alpha(\zeta) \cdot (\bar b ,h) \, = \, \big(u_\alpha\big(\tilde \alpha(h^{-1}) ( \zeta^{q_{\alpha}}-\zeta)\big)\bar b, v_\alpha (\zeta) h\big).$}

\noindent Let us show that this is a well-defined action on $\Yt$: consider $(\bar b,h) \in \Yt$ and $(\bar b',h') = v_\alpha(\zeta) \cdot (\bar b, h)$ their image by $\zeta$. Using the fact that $\mathcal{L}_\Delta (v_\alpha(\zeta)) = u_\alpha (\zeta^{q_{\alpha}}-\zeta)$, we compute in the group $\B_\Delta$:

\centers{$ \begin{array}{r@{\ \, = \ \, }l} \mathcal{L}_\Delta(h')  & {h'}^{-1} F(h') \, = \,   h^{-1} \mathcal{L}_\Delta(v_\alpha(\zeta)) F(h) \\[4pt]
%& h^{-1} u_\alpha(\zeta^{q_{\alpha}}-\zeta) F(h) \\[4pt]
& h^{-1} u_\alpha(\zeta^{q_{\alpha}}-\zeta)h \,h^{-1} F(h) \\[4pt]
\mathcal{L}_\Delta(h') & u_\alpha\big(\tilde \alpha (h^{-1}) (\zeta^{q_{\alpha}}-\zeta) \big) \mathcal{L}_\Delta(h) 
  \end{array}$}
  
\noindent so that $\mathcal{L}_\Delta(h') = {\bar b'}_\Delta$ by definition of $\bar b'$. Moreover, the one-parameter subgroup $\U_\alpha$ is contained in ${}^x \U$ since $x^{-1} (\alpha) > 0$, and hence $\bar b' \in {}^x\U \bar b \subset x(\U \dot w \U)F(x)^{-1}$, which proves that $(\bar b',h') \in \Yt$.

\sk

To conclude, we remark that restricting this action to $V_\alpha$ amounts to restrict the parameter $\zeta$ to $\F_{q_\alpha}$. In that case, we clearly recover the natural action of $V_\alpha$ coming from the action of $D(\U)^F \backslash U$ we described before.
 \end{proof}

As an application, we show that, as expected in \cite[Conjecture 2.7]{BR}, the regular characters do not occur in the cohomology of the variety $\Yxw$. This is the first step for determining the isotypic part of these representations in the cohomology of the variety $\Y(\dot w)$, and hence proving theorem \ref{gg}.

\begin{cor}\label{cor1}Let $\psi : U \longrightarrow \Lambda^\times$ be a $G$-regular linear character, and $x$ an element of the Weyl group $W$ different from $w_0$. Then 

\centers{$ e_\psi \R(\Yxw,\Lambda) \simeq 0$}

\noindent in the derived category $\mathcal{D}^b(\mathrm{mod}$-$\Lambda \T^{wF})$.

\end{cor}

\begin{proof} By formula \ref{tild}, we can replace $\Yxw$ by $\Yt$. Since $x$ is different from $w_0$, there exists a positive simple root $\alpha$ such that $x^{-1}(\alpha) > 0$; by proposition \ref{ext}, the action of $V_\alpha$ on  $D(\U)^F\backslash \Yt$  extends to an action of the connected group $\V_\alpha$.
%, so that the induced action on $\R (D(\U)^F\backslash \Yt)$ is trivial (see ?? prop. ??).  
On the other hand, $\psi$ is a regular character, and hence its restriction to $V_\alpha$ is non-trivial. 
Therefore, we obtain by proposition \ref{coho}.(iii):

%the $\psi$-isotypic component of this complex is quasi-isomorphic to zero and we get
\vskip-6pt
\centers{$ e_\psi \R(D(\U)^F \backslash \Yt, \Lambda)  \, \simeq \, e_\psi \big(\Lambda {\ol}_{\Lambda V_\alpha} \, \R(D(\U)^F \backslash \Yt, \Lambda) \big)\, \simeq 0 $}

\noindent  Finally, since $\psi$ is trivial on $D(\U)^F$, we get
\vskip-3pt
\centers{$ e_\psi \R(\Yt, \Lambda)  \, \simeq \, \Lambda {\ol}_{\Lambda D(\U)^F} \, e_\psi \R(\Yt, \Lambda)  \, \simeq \, e_\psi \R(D(\U)^F \backslash \Yt, \Lambda) \, \simeq 0 $}

\noindent using the fact that $e_\psi$ is central in $\Lambda U$ and the proposition \ref{coho}.(ii).\end{proof}

\bk

\noindent \textbf{2 - The variety $\Yo$}\mk

In this section we are concerned with the cohomology of the Deligne-Lusztig variety $\Yo$ corresponding to the maximal Schubert cell. Our aim is to determine the contribution of the regular characters in these cohomology groups. For this purpose, we define a partition of the variety $\Xo$, coming from  Deodhar's decomposition of the double Schubert cell $\B w \cdot \B \cap \B^- \cdot \B$ and we show that these specific characters occur only in the cohomology of the maximal cell. We deduce the result at the level of $\Y$ using the same method as in \cite{BR}.\sk

We keep the basic assumptions and the notations of the section \ref{ddec}: $w=s_1 \cdots s_r$  is a reduced expression of $w$, $\gamma \in \Gamma_1$ is a distinguished subexpression of $w$ ending by $1$ and $D_\gamma$ denotes the corresponding Deodhar cell in $\B w \cdot \B \cap \B^- \cdot \B$. We will be interested in the pull-back $\Omega_\gamma$ of $D_\gamma$ in $\U^-$, for which we have a nice parametrization (see notations \ref{not}).  In these terms, Deodhar's decomposition can be written as

\centers{$ \B w \B \cap \U^- \, = \, \displaystyle \coprod_{\gamma \in \Gamma_1} \Omega_\gamma$}

%More precisely, if we define $\Omega_\gamma$ to be the inverse image by the isomorphism $u \in \U^- \longmapsto \B^- \cdot \B$ then every element of $\Omega_\gamma$ can be written in 

\noindent Now since $F(w_0) = w_0$, the maximal piece of the Deligne-Lusztig variety $\X(w)$ has the following expression (see section \ref{prel}.4):

\centers{$ \begin{array}{rl} \Xo \, = & \hskip -2mm \big\{g\B \in \B w_0 \cdot \B \ \big| \ g^{-1}F(g) \in \B w \B \big\} \\[4pt]
\simeq & \hskip -2mm \big\{u \in \U\ \big| \ {}^{w_0} (u^{-1}F(u)) \in \B w \B\cap \U^- \big\} \end{array}$}

\noindent so that if we define the piece $\X_\gamma$ to be:

\centers{$\X_\gamma \, = \, \big\{u \in \U \ \big| \  {}^{w_0} (u^{-1}F(u))  \in \Omega_\gamma\big\}$}

\leftcenters{then we get}{$ \Xo \, \simeq \,  \displaystyle \coprod_{\gamma \in \Gamma_1} \X_\gamma.$}

\noindent Note that each component of this partition is stabilized by the action of $U$, and that this isomorphism is $U$-equivariant. The same decomposition holds also for $\Yo$ if we define $\Y_\gamma$ to be the pullback $\pi_w^{-1}(\X_\gamma)$ where $\pi_w: \Y(\dot w) \longrightarrow \X(w)$ is the quotient map by the right action of $\T^{wF}$ defined in section \ref{prel}.3.  One can easily check that this defines a family of locally closed smooth subvarieties of $\Yo$ satisfying the filtration property of lemma \ref{filt}.

\mk

\begin{example}\ Let us go back to our example of $\mathrm{GL}_3(\F)$, endowed with the standard $\F_q$-structure (see example \ref{ex}). Given an element $u$ of $\U$ written as

\centers{$u \, =  \, \left( \begin{array}{ccc} 1 & a & c\\ 0 & 1 & b \\ 0 & 0& 1 \end{array} \right)$}

\leftcenters{one has}{$ {}^{w_0} (u^{-1}F(u)) \, = \, \left( \begin{array}{ccc} 1 & 0 & 0\\ b^q-b & 1 & 0 \\ c^q-c-a(b^q-b) & a^q -a & 1 \end{array} \right).$}

\noindent Together with the description of  $\B w_0 \B \cap \U^-$ given in the example \ref{ex}, we obtain

\centers{$ \X_{w_0}(w_0) \, \simeq \, \left\{ (a,b,c) \in (\G_a)^3 \, \Big| \, \begin{array}{l} c^q-c -a(b^q-b) \neq 0 \\[3pt] c^q-c -a^q(b^q-b) \neq 0 \end{array} \right\}.$}

\noindent Since we are concerned with regular characters only, it is more convenient to consider the quotient variety $\db \X_{w_0}(w_0)$. But taking the quotient by $D(\U)^F \simeq \F_q$ amounts to take the expression $C=c^q-c-a(b^q-b)$ as a variable, so that 

\centers{$ \db \X_{w_0}(w_0) \, \simeq \, \big\{ (a,b,C) \in (\G_a)^2\times \G_m \, \big| \, C-(a^q-a)(b^q-b) \neq 0 \big\}.$}

\noindent It remains to apply the partition of $\B w_0 \B \cap \U^-$ given in the example \ref{ex} to deduce a decomposition of this quotient into: \begin{itemize}

\item a closed subvariety $\X_C$ defined by the equation $a^q-a =0$. From the above description, it is clearly isomorphic to $\F_q \times \G_a \times \G_m$ with the product action of $\db U \simeq \F_q \times \F_q$ on the first two coordinates;

\item an open subvariety $\X_O$ defined by the equation $a^q-a \neq 0$. By the change of variables $C' = C/(a^q-a)$, it becomes isomorphic to

\centers{$\begin{array}{r@{\ \, \simeq \ \, }l} \X_O & \mathcal{L}^{-1}(\G_m) \times \big\{ (b,C') \in \G_a \times \G_m \ \big| \ C' \neq b^q-b \big\} \\[4pt] & \mathcal{L}^{-1}(\G_m) \times \big( (\G_a \times \G_m) \setminus \mathcal{L}^{-1}(\G_m) \big)\end{array}$}

\noindent where we have embedded $\mathcal{L}^{-1}(\G_m)$ into $\G_a \times \G_m$ by the map $b \longmapsto (b,b^q-b)$.

\end{itemize}

\noindent  It is now a straightforward matter to compute the cohomology groups of the varieties and it turns out that the cohomology of $\X_C$ does not carry any regular character whereas there is one occurring in the cohomology of the maximal cell $\X_O$. More precisely, for a regular character $\psi$, we get
%\centers
{$ e_\psi \H^\bullet (\db \X_{w_0}(w_0)) \, \simeq \, e_\psi \H^\bullet(\X_O) \, \simeq \, \Lambda_\psi[-3]$}
%\noindent 
as expected in theorem \ref{br}. \end{example}

\mk

Keeping in mind the previous example, we give now a description of the quotient $D(\U)^F \backslash \X_\gamma$ in terms of some combinatorial data, namely the sets $I(\gamma)$ and $J(\gamma)$ introduced in section \ref{ddec}. In oder to state this result, we need to introduce some more notations; for $m,n \in \mathbb{N}$ and $q$ a power of $p$, we define the following variety:

\centers{$ \X_q(n,m) \, = \, \big\{(\zeta, \mu_1, \ldots, \mu_n, \lambda_1, \ldots, \lambda_m) \in (\G_a)^{n+1}  \times (\G_m)^m \ \big| \ \zeta^{q} - \zeta = \mathop{\sum}\limits_{i} \mu_i + \mathop{\sum}\limits_{j} \lambda_j \big\}$} 

\noindent which is endowed with an action of $\F_q$ by translation on the first variable. 

\begin{prop} Let $\gamma$ be a distinguished subexpression of $w$. For $\alpha \in \Delta$ a simple positive root and $\mathcal{O}_\alpha$ the corresponding orbit in $\Delta / \phi$, we consider the following integers:
\begin{itemize} 

\item $n_\alpha(\gamma) = \big| \{i = 1, \ldots, r \ | \ w_0(\tilde \beta_i) \in  \mathcal{O}_\alpha \ \, \text{and} \, \ i \in I(\gamma) \setminus J(\gamma) \} \big|; $

\item $m_\alpha(\gamma) = \big| \{i = 1, \ldots, r \ | \ w_0(\tilde \beta_i) \in  \mathcal{O}_\alpha \ \, \text{and} \, \ i \notin I(\gamma) \} \big| $;

\item $\bar n(\gamma) = |I(\gamma)| -|J(\gamma)|  - \sum n_\alpha(\gamma)$;

\item $\bar m(\gamma) = r-|I(\gamma)| - \sum m_\alpha(\gamma)$.

\end{itemize}

\sk

\noindent Then there exists a $D(\U)^F \backslash U$-equivariant equivalence of \'etale sites

\centers{$ D(\U)^F \backslash \X_\gamma \, \cong (\G_a)^{\bar n(\gamma)} \times (\G_m)^{\bar m(\gamma)} \times \hskip -2mm \displaystyle \prod_{\alpha \in [\Delta/\phi]} \X_{q_\alpha}(n_\alpha(\gamma), m_\alpha(\gamma))$}

\noindent where $D(\U)^F \backslash U \simeq \prod_{\alpha \in [\Delta/\phi]} V_\alpha$ acts on the product  $\prod_{\alpha \in [\Delta/\phi]} \X_{q_\alpha}(n_\alpha(\gamma), m_\alpha(\gamma))$ via the identifications $V_\alpha \simeq \F_{q_\alpha}$.  Moreover, if $(\G,F)$ is split, this is an isomorphism of varieties.
\end{prop}

\begin{proof} Recall that we have chosen in section \ref{prel}.2 an isomorphism of groups between $D(\U) \backslash \U$ and the abelian group $\prod_{\alpha \in \Delta} \, \U_\alpha \subset \B_\Delta$. As in the proof of proposition \ref{ext}, we can hence realize the quotient variety as:

\centers{$ D(\U)^F \backslash \Xo \, \simeq \, \big\{(\bar u,h) \in \U \times \mathop{\prod}\limits_{\alpha \in \Delta} \U_\alpha \ \big| \ \bar u_\Delta  = \mathcal{L}_{\Delta}(h) \ \, \mathrm{and} \ {}^{w_0} \bar u \in \B w \B 
 \big\}. $}
 
 \noindent The restriction of this isomorphism to $D(\U)^F \backslash \X_\gamma$ gives the following description:
 
\centers{$ \begin{array}{r@{\ \, \simeq \ \,}l} D(\U)^F \backslash \X_\gamma & \big\{(\bar u,h) \in \U \times \mathop{\prod}\limits_{\alpha \in \Delta} \U_\alpha \ \big| \ \bar u_\Delta  = \mathcal{L}_{\Delta}(h) \ \, \mathrm{and} \ {}^{w_0} \bar u \in \Omega_\gamma  \big\} \\[10pt]
&  \big\{(\bar u,h) \in {}^{w_0} \Omega_\gamma \times \mathop{\prod}\limits_{\alpha \in \Delta} \U_\alpha \ \big| \ \bar u_\Delta  = \mathcal{L}_{\Delta}(h) \big\}. \end{array}  $}
 
 \noindent According to the parametrization of the cell $D_\gamma$ (see notations \ref{not}), every element $\bar u$ of ${}^{w_0} \Omega_\gamma$ can be uniquely written as
 
 \centers{$ \bar u \, = \, \displaystyle \prod_{\beta \in \Phi(\gamma)} u_{w_0(\beta)}(x_\beta) \ \in \U$}
 
 \noindent with the variables $x_\beta$ living in ${\G_a}$ or ${\G_m}$ whether $\beta$ is of the form  $\tilde \beta_i$ with $i$ belonging to $I(\gamma)$ or not. Now calculating $\bar u_\Delta$ (or equivalently $\bar u D(\U)$) amounts to keep only the positive simple roots occurring in this expression, that is the negative roots $\beta$ for which $w_0(\beta) \in \Delta$. On the other hand, for any element $h =  \big(u_{\alpha}(\zeta_\alpha)\big)_{\alpha \in\Delta}$ of $\prod_{\alpha \in \Delta} \U_\alpha$, one has
 
 \centers{$ \mathcal{L}_\Delta(h) \, = \, h^{-1} F(h) \, = \, \big(u_{\phi( \alpha)}(\zeta_{\alpha}^{q_\alpha^\circ}-\zeta_{\phi(\alpha)}) \big)_{\alpha \in \Delta}. $}

\noindent Therefore, the quotient variety $D(\U)^F \backslash \X_\gamma$ has a system of coordinates given by the two sets of variables $(x_\beta)$ and $(\zeta_\alpha)$ satisfying the relations

\centersright{$\forall \, \alpha \in \Delta \qquad \quad \zeta_{\alpha}^{q_\alpha^\circ}-\zeta_{\phi(\alpha)} \, = \hskip-2mm \SUM{w_0(\beta) = \phi(\alpha)}{}  \hskip-1.5mm x_\beta$.}{}

%\noindent From this set of identities, one can deduce that the value of the scalar $\zeta_\alpha$ is  uniquely determined by any $\zeta_{\delta}$ for $\delta$ a simple root in the orbit $\mathcal{O}_\alpha$ of $\alpha$ and the variables $(x_\beta)$. Moreover, one can check

\noindent Note that we can group the relations and the variables according to the class of the simple root $\alpha$ in $\Delta/\phi$ which is involved in the equations. One deduces that $D(\U)^F \backslash \X_\gamma$ decomposes into a product of varieties indexed by $[\Delta/\phi]$. Without any loss of generality, we can hence assume that $\phi$ acts transitively on $\Delta$.  In order to simplify the  notations, we fix a simple root $\alpha$ and we will denote $\zeta_{\phi^i(\alpha)}$ by $\zeta_i$ and $q_{\phi^i(\alpha)}^\circ$ by $q_i^\circ$. The relations defining $\X_\gamma$ can be then rewritten as 

\centers{$ \forall \, i = 1,\ldots, d \qquad \quad \zeta_{i} \, = \,  \zeta_{i-1}^{q_{i-1}^\circ} -  \hskip-2mm \SUM{w_0(\beta) = \phi^i(\alpha)}{}  \hskip-1.5mm x_\beta$}

\noindent It shows that the value of  $\zeta_{i}$ is  uniquely determined by $\zeta_{0}$ and the variables $(x_\beta)$. Moreover, by substitution, the last equation becomes 

\centers{$ \zeta_0^{q_\alpha} - \zeta_0 \, =  \hskip-0.5mm \SUM{w_0(\beta) = \alpha}{} x_\beta \ + \hskip-2.5mm \SUM{w_0(\beta) = \phi^{d-1}(\alpha)}{} \hskip-2mm x_\beta^{q_{d-1}^\circ}  \ + \ \cdots \ +\ \hskip-2mm \SUM{w_0(\beta) = \phi(\alpha)}{} \hskip-1.5mm x_\beta^{q_{d-1}^\circ \cdots q_2^\circ q_1^\circ}. $}

\noindent We should not forget that some of the variables $x_\beta$ are not involved in this equation, namely the ones for which $w_0(\beta)$ is not a simple root; they correspond to the factor $ (\G_a)^{\bar n(\gamma)} \times (\G_m)^{\bar m(\gamma)}$. Indeed, if we set $q_\beta = q_{d-1}^\circ \cdots q_i^\circ$ provided that $w_0(\beta) = \phi^i(\alpha)$ for $i \in \{1,\ldots,d\}$, we obtain the following description of the quotient variety $D(\U)^F \backslash \X_\gamma$:

\centers{$D(\U)^F \backslash \X_\gamma \, \simeq \,  (\G_a)^{\bar n(\gamma)} \times (\G_m)^{\bar m(\gamma)} \times \big\{(\zeta,(x_\beta)_{\beta \in -\Delta}) \ \big| \ \zeta^{q_\alpha} - \zeta = \sum x_\beta^{q_\beta} \big\}$.}

\noindent 
Finally, up to a new labelling of the variables $(x_\beta)_{\beta \in -\Delta}$ whether they belong to $\G_a$ or $\G_m$, the map $(x_\beta) \longmapsto (x_\beta^{q_\beta})$ induces the expected equivalence  of \'etale sites. 
%\centers{$D(\U)^F \backslash \X_\gamma  \cong  (\G_a)^{n(\gamma)} \times (\G_m)^{m(\gamma)} \times \X_{q_\alpha}(n_\alpha(\gamma),m_\alpha(\gamma))$.}
\end{proof}

%From now on, we assume that $\G$ is semi-simple and simply connected, and that $w$ is not contained in a proper $F$-stable standard parabolic subgroup of $W$. In that case, the Deligne-Lusztig varieties $\X(w)$ and $\Y(\dot w)$ are both irreducible (see prop ?? and ??).  

We want to lift the above description of $\X_\gamma$ up to the variety $\Y_\gamma$. Unfortunately, one cannot deduce directly from the previous proposition a parametrization of $\Y_\gamma$ since the quotient map by $\T^{wF}$ might not split. Following \cite{BR}, we shall nevertheless construct an abelian covering of $D(\U)^F \backslash \Y_\gamma$ which will have the expected shape.
\mk

\begin{notations} For the sake of coherence with the notations of the last proposition, we introduce the positive integers $m(\gamma) = r-|I(\gamma)|$ and $n(\gamma) = |I(\gamma) |- |J(\gamma)|$ so that the Deodhar cell is isomorphic to $(\G_a)^{n(\gamma)} \times (\G_m)^{m(\gamma)}$.   The product variety given in the same proposition will be simply denoted by $\X_\gamma'$.\end{notations}

\mk

Let us consider the pullback in $\Y_\gamma$ of a connected component of $U\backslash \Y_\gamma$ and denote it by $\Yg$. By construction, $\Yg$ is stable by $U$ and the quotient $\ub \Yg$ is connected (whereas $\Yg$ might not be). Besides, since $\ub \Y_\gamma / \T^{wF} \simeq \ub \X_\gamma$ is connected, the group $\T^{wF}$ permutes transitively the connected components of $\ub \Y_\gamma$. Therefore, if we define the group $H$ to be the stabilizer of $\ub \Yg$ in $\T^{wF}$,  then the multiplication induces the following $(\T^{wF})^{\mathrm{op}}$-equivariant isomorphism of varieties:

\centers{$ (\ub \Yg) \times_H \T^{wF} \, \simeq \, \ub \Y_\gamma$.}

\noindent On the other hand, the actions of $U$ and $\T^{wF}$ commute, so that we can also check that

\leftcentersn{}{
%$ \Yg \times_H \T^{wF} \, \simeq \, \Y_\gamma$ \qquad and \qquad 
$ (\db \Yg) \times_H \T^{wF} \, \simeq \, \db \Y_\gamma$.}{yo}

We now define the analog $\Y_\gamma'$ of $\X_\gamma'$ for the variety $\db \Yg$, which fits into the following commutative diagram, where all the squares are cartesian:

%We first consider the analog $\Y_\gamma'$ of $\X_\gamma'$ for the variety  $D(\U)^F \backslash \Y_\gamma$, which fits into the following diagram, where all the squares are cartesian:

%\sk

%\centers{$\xymatrix@C=1.5cm{\Y_\gamma' \ar[r]^{} \ar@{->>}[d] & D(\U)^F \backslash \Yg \ar@{->>}[d]  \ar@{->>}[r]  &  U \backslash \Yg  \ar@{->>}[d] \\ \X_\gamma' \ar[r] & D(\U)^F \backslash \X_\gamma  \ar@{->>}[r] & U \backslash \X_\gamma} $}

\sk 

\centers{$ \begin{psmatrix}[colsep=2cm] \Y_\gamma'  &  D(\U)^F \backslash \Yg &  U \backslash \Yg  \\ \X_\gamma'  &  D(\U)^F \backslash \X_\gamma &  U \backslash \X_\gamma  
\psset{arrows=->>,nodesep=3pt} 
\everypsbox{\scriptstyle} 
\ncline[arrows=<->]{1,1}{1,2}^{\hskip-4mm \mathrm{equ}}  \ncline{1,2}{1,3}  
\ncline[arrows=<->]{2,1}{2,2}^{\hskip-4mm \mathrm{equ}}  \ncline{2,2}{2,3}
\ncline{1,1}{2,1}<{\pi_w'}  \ncline{1,2}{2,2}<{\pi_w^\circ} \ncline{1,3}{2,3}<{\pi_w^\circ}  
 \end{psmatrix}$}

\sk

\noindent In this diagram and the following ones, the notation $\mathrm{equ}$ stands for an equivalence of \'etale sites. Note that the morphims $\pi_w'$ and $\pi_w^\circ$ are isomorphic to quotient maps by the right action of $H \subset \T^{wF}$, which is a $p'$-group. In particular, the map $\pi_w^\circ : \ub \Yg \longrightarrow (\G_a)^{n(\gamma)}  \times (\G_m)^{m(\gamma)}$ is a Galois covering which is tamely ramified. By Abhyankar's lemma (see \cite[Expos\'e XIII, 5.3]{SGA}) there exists a Galois covering $\varpi : (\G_a)^{n(\gamma)}  \times (\G_m)^{m(\gamma)} \longrightarrow \ub \Yg$ with group $N$ such that the composition $\pi_w^\circ \circ \varpi : (\G_a)^{n(\gamma)}  \times (\G_m)^{m(\gamma)} \longrightarrow (\G_a)^{n(\gamma)}  \times (\G_m)^{m(\gamma)}$ sends $(\mu_1, \ldots, \mu_{n(\gamma)},\lambda_1, \ldots, \lambda_{m(\gamma)})$ to $(\mu_1, \ldots, \mu_{n(\gamma)},\lambda_1^d, \ldots, \lambda_{m(\gamma)}^d)$ for some positive integer $d$ relatively prime to $p$. We summarize the different maps involved in the following diagram:

%\noindent By proposition ??, the map $\Y_\gamma' \longrightarrow  D(\U)^F \backslash \Y_\gamma$ is an equivalence of etale sites.

%\noindent Since the variety $U \backslash \X_\gamma$ is connected, the group $\T^{wF}$ acts transitively on the set of connected components of $U \backslash \X_\gamma$. Therefore, if we consider one particular component $Y_\gamma^\circ$ and if we denote by $H$ its stabilizer in $\T^{wF}$, the multiplication induces a $(\T^{wF})^{\mathrm{op}}$-equivariant isomorphism of varieties 
%: \centers
%{$  \Y_\gamma^\circ \times_H \T^{wF} \, \simeq \, U \backslash   \Y_\gamma.$} Now the map $\Y_\gamma^\circ \longrightarrow (\G_a)^{n(\gamma)} \times (\G_m)^{m(\gamma)}$ is a Galois covering of group $H$.

%\sk

\centers{$ \begin{psmatrix} [colsep=5.3mm] [mnodesize=0pt]  (\G_a)^{n(\gamma)} \times (\G_m)^{m(\gamma)}   \psspan{3} \\ & & U \backslash \Y_\gamma^\circ   \\ & & U \backslash \X_\gamma  & 
 \hskip -3.8mm \simeq  \ (\G_a)^{n(\gamma)} \times (\G_m)^{m(\gamma)}
\psset{arrows=->>,nodesep=3pt} 
\everypsbox{\scriptstyle} 
\ncline{1,1}{2,3}<{\varpi}>{/N} \ncline{2,3}{3,3}<{\pi_w^\circ}>{/H} \ncarc[arcangle=20]{1,1}{3,4}>{/ (\mathbf{\bbmu}_d)^{m(\gamma)}}
\end{psmatrix}$}

\noindent where $\mathbf{\bbmu}_d$ denotes the group of the $d$-th roots of unity in $\F$. In this setting, $N$ is a subgroup of $(\mathbf{\bbmu}_d)^{m(\gamma)}$ and we have a canonical group isomorphism $(\mathbf{\bbmu}_d)^{m(\gamma)} / N \simeq H$. If we form the fiber product of $\Y_\gamma'$ and $ (\G_a)^{n(\gamma)} \times (\G_m)^{m(\gamma)} $ above $\ub \Yg$, we obtain the following diagram whose squares are cartesian:

%\centers{$ \xymatrix@C=-20mm{  \big((\G_a)^{n(\gamma)} \times (\G_m)^{m(\gamma)} \big) \times_H \T^{wF} \ar@{->>}[d]  \ar@/^15pt/[ddr] \\ U \backslash \Y_\gamma  \ar@{->>}[d]  \\  \hskip 35mm U \backslash \X_\gamma  \, \simeq \, (\G_a)^{n(\gamma)} \times (\G_m)^{m(\gamma)}& \begin{array}{l}  {}^{}\end{array}    }$}

\sk

\leftcentersn{}{$ \hskip-3cm \begin{psmatrix}
[colsep=5.3mm]
[mnodesize=2cm]  \Y_\gamma''    & &  &  [mnodesize=0cm] 
 (\G_a)^{n(\gamma)} \times (\G_m)^{m(\gamma)}\psspan{3} 
 &  \\[0pt]  [mnodesize=2cm]  \Y_\gamma'    & &    [mnodesize=2.4cm] D(\U)^F \backslash \Y_\gamma^\circ & & &  U \backslash \Y_\gamma^\circ  
 \\[0pt] [mnodesize=2cm]  \X_\gamma'    & &   [mnodesize=2.4cm] D(\U)^F \backslash \X_\gamma & & & U \backslash \X_\gamma  & 
 [mnodesize=0pt] \hskip -3.8mm \simeq \ (\G_a)^{n(\gamma)} \times (\G_m)^{m(\gamma)}
\psset{arrows=->>,nodesep=3pt} 
\everypsbox{\scriptstyle} 
\ncline{1,4}{2,6}<{\varpi}>{/ N}  \ncarc[arcangle=25]{1,4}{3,7}>{/ (\mathbf{\bbmu}_d)^{m(\gamma)}}
\ncline[arrows=<->]{2,1}{2,3}^{\hskip-4mm \mathrm{equ}}  \ncline{2,3}{2,6}  
\ncline[arrows=<->]{3,1}{3,3}^{\hskip-4mm \mathrm{equ}}  \ncline{3,3}{3,6}
\ncline{1,1}{1,4} \ncline{1,1}{2,1}>{/N}
\ncline{2,1}{3,1}<{\pi_w'}  \ncline{2,3}{3,3}<{\pi_w^\circ} \ncline{2,6}{3,6}<{\pi_w^\circ}>{/H}  
\end{psmatrix}$}{diag}

\sk

%\centers{$ \xymatrix{  & & \big((\G_a)^{n(\gamma)} \times (\G_m)^{m(\gamma)} \big) \times_H \T^{wF} \ar@{->>}[d]  \ar@/^15pt/[dddr] \\ \Y_\gamma' \ar[r]^{} \ar@{->>}[d] & D(\U)^F \backslash \Y_\gamma \ar@{->>}[d]  \ar@{->>}[r]  & U \backslash \Y_\gamma  \ar@{->>}[d]  \\ \X_\gamma' \ar[r] & D(\U)^F \backslash \X_\gamma  \ar@{->>}[r] & U \backslash \X_\gamma  \\ & & &  (\G_a)^{n(\gamma)} \times (\G_m)^{m(\gamma)}    }$}

\noindent so that by definition of $\X_\gamma'$ and the properties of the Galois covering $\pi_w^\circ \circ \varpi$ we get a $\db U \times \big((\mathbf{\bbmu}_d)^{m(\gamma)}\big)^{\mathrm{op}}$-equivariant isomorphism of varieties:

\leftcentersn{}{$\Y_\gamma'' \, \simeq \, (\G_a)^{\bar n(\gamma)} \times (\G_m)^{\bar m(\gamma)} \displaystyle \prod_{\alpha \in [\Delta/\phi]} \Y_{q_\alpha,d}(n_\alpha(\gamma),m_\alpha(\gamma))$}{ypp}

\noindent where the factors on the right-hand side are defined by

\leftcenters{}{$  \Y_{q,s}(n,m) \, = \, \big\{(\zeta, \mu_1, \ldots, \mu_n, \lambda_1, \ldots, \lambda_m) \in (\G_a)^{n+1}  \times (\G_m)^m \ \big| \ \zeta^{q} - \zeta = \mathop{\sum}\limits_{i} \mu_i + \mathop{\sum}\limits_{j} \lambda_j^s \big\}$}

\noindent and endowed with a natural action $\F_q $ by translation on the first variable together with the action of $(\mathbf{\bbmu}_s)^m $ obtained by multiplication on $(\G_m)^m$.  Up to an equivalence of \'etale sites, we have performed in this way a construction of an abelian covering of the variety $\db \Y_\gamma^\circ$ which decomposes into a product of varieties. Using the results of \cite{BR}, we shall first compute the cohomology of each of these factors:

\begin{lem}\label{brl}Let $\psi$ be a non trivial character of $\F_q$. Then 

\centers{$ e_\psi \R (\Y_{q,s}(n,m),\Lambda) \, \simeq  \  \left\{ \hskip-1.3mm \begin{array}{l} \Lambda_\psi (\mathbf{\bbmu}_s)^m [-m] \ \ \text{if } n = 0\text{,}\\[4pt]
0 \ \ \text{otherwise} \end{array}\right.$}

\noindent in the derived category $\mathcal{D}^b(\Lambda \F_q$-$\mathrm{mod}$-$\Lambda (\mathbf{\bbmu}_s)^m)$.
\end{lem}

\begin{proof} If $n \neq 0$, the equation defining $\Y_{q,s}(n,m)$ can be rewritten as 

\centers{$ \mu_1 \, = \, \zeta^q - \zeta - \displaystyle \mathop{\sum}\limits_{i=1}^m \lambda_i - \mathop{\sum}\limits_{j=2}^n \mu_j$}

\noindent so that $\Y_{q,s}(n,m) \simeq (\G_a)^n \times (\G_m)^m$, with an action of $\F_q$ on the first coordinate. Since the cohomology of the affine line is given by $\R(\G_a,\Lambda) \simeq \Lambda[-2]$ we get the result by the K\"unneth formula (see proposition \ref{coho}.(i)).

\sk

The variety $\Y_{q,s}(0,m)$ can be completely described by the curves $\Y_{q,s} = \Y_{q,s}(0,1)$ studied by Bonnaf\'e and Rouquier in \cite{BR} (see also \cite{Lau}): the map

\centers{$ \begin{array}{rcl} \Y_{q,s} \times \Y_{q,s} \times \cdots \times \Y_{q,s} & \longrightarrow & \Y_{q,s}(0,m) \\[4pt] \big((\zeta_1,\lambda_1), \ldots ,(\zeta_m,\lambda_m)\big) & \longmapsto & (\zeta_1 + \cdots + \zeta_m,\lambda_1, \ldots, \lambda_m) \end{array}$}

\noindent induces indeed the following $\F_q\times (\mathbf{\bbmu}_s)^m$-equivariant isomorphism of varieties:

\centers{$ \Y_{q,s} \times_{\F_q} \times \Y_{q,s} \times_{\F_q} \cdots \times_{\F_q} \Y_{q,s} \ \simeq \ \Y_{q,s}(0,m)$.}

\noindent To conclude, it remains to translate this isomorphism in the category $\mathcal{D}^ b(\Lambda \F_q$-$\mathrm{mod}$-$\Lambda (\mathbf{\bbmu}_s)^m)$; we can then deduce the result from the case $m= 1$ which was solved in \cite[Lemma 3.6]{BR}:

\vskip-3pt \centers{$ \begin{array}{r@{\, \ \simeq \, \ }l} e_\psi\R(\Y_{q,s}(0,m), \Lambda) & e_\psi\R(\Y_{q,s},\Lambda) {\ol}_{\Lambda \F_q} \cdots \ {\ol}_{\Lambda \F_q}e_\psi\R(\Y_{q,s}, \Lambda) \\[4pt]
& \Lambda_\psi \mathbf{\bbmu}_s[-1] {\ol}_{\Lambda \F_q} \cdots \ {\ol}_{\Lambda \F_q}\Lambda_\psi \mathbf{\bbmu}_s[-1] \ \simeq \ \Lambda_\psi (\mathbf{\bbmu}_s)^m[-m].\end{array}$}

\noindent the first quasi-isomorphism coming also from the fact that $e_\psi$ is an idempotent.
\end{proof}

We have now at our disposal all the ingredients we need to compute the isotypic part of a regular character in the cohomology of each variety $\Y_\gamma$:

\begin{prop} Let $\gamma$ be a distinguished subexpression of $w$ ending by 1, and $\psi : U \longrightarrow \Lambda^\times$ a $G$-regular linear character. Then

\centers{$ e_\psi \R (\Y_\gamma,\Lambda) \, \simeq  \  \left\{ \hskip-1.3mm \begin{array}{l} \Lambda \T^{wF} [-\ell(w)] \ \  \text{if } \gamma=(1,1,\ldots,1) \text{,} \\[4pt]
0 \ \ \text{otherwise} \end{array}\right.$}

\noindent in the derived category $\mathcal{D}^b(\mathrm{mod}$-$\Lambda \T^{wF})$.

\end{prop}

\begin{proof} Since $\psi$ is trivial on $D(\U)^F$, we can argue as in the proof of the corollary \ref{cor1} to show that it is sufficient to prove the result for the quotient variety $\db \Y_\gamma$. Furthermore,  this variety can be expressed by formula \ref{yo} as an amalgamated product with $\T^{wF}$ above $H$ so that by formula \ref{amalg} one obtains

\vskip -7pt\centers{$  \R(\db \Y_\gamma, \Lambda) \, \simeq \, \R(\db \Y_\gamma^\circ, \Lambda) \ {\ol}_{\Lambda H} \, \T^{wF}.$}

\noindent We continue the reductions: from the commutative diagram \ref{diag}, we have an equivalence of \'etale sites between $\db \Yg$ and $\Y_\gamma'$ and an isomorphism between $\Y_\gamma'$ and $\Y_\gamma'' /N$, both being $U\times H^{\mathrm{op}}$-equivariant.  Together, they induce the following quasi-isomorphism:

\vskip -5pt \centers{$ \R(\db \Yg, \Lambda) \, \simeq \, \R(\Y_\gamma'', \Lambda) \ {\ol}_{\Lambda N} \, \Lambda$.}

\noindent Since $H \simeq (\mathbf{\bbmu}_d)^m/N$, we deduce by an adjunction formula that 

\centers{$\begin{array}{r@{\ \, \simeq \ \, }l}  \R(\db \Y_\gamma, \Lambda) &  \big(\R(\Y_\gamma'', \Lambda) \ {\ol}_{\Lambda N} \, \Lambda \big) \ {\ol}_{\Lambda H} \, \Lambda \T^{wF} \\[4pt]
& \big(\R(\Y_\gamma'', \Lambda) \ {\ol}_{\Lambda (\mathbf{\bbmu}_d)^m} \, \Lambda H \big) \ {\ol}_{\Lambda H} \, \Lambda \T^{wF} \\[4pt]
 \R(\db \Y_\gamma, \Lambda)  &  \R(\Y_\gamma'', \Lambda) \ {\ol}_{\Lambda (\mathbf{\bbmu}_d)^m} \, \Lambda \T^{wF} \end{array} $}

\noindent Denote by $\psi_\alpha$ the restriction of $\psi$ to the group $V_\alpha$. The decomposition of the variety $\Y_\gamma''$ (see formula \ref{ypp}) translates, by the K\"unneth formula, into the following quasi-isomorphisms:

\vskip -1pt
\centers{ $\begin{array}{r@{\, \ \simeq \, \ }l} e_\psi \R(\Y_\gamma'', \Lambda) & \R\big( (\G_a)^{\bar n(\gamma)} \times (\G_m)^{\bar m(\gamma)}, \Lambda\big) \ \ol \ \displaystyle e_\psi \R\Big(\hskip -3mm \prod_{\alpha \in [\Delta/\phi]} \hskip-2mm \Y_{q_\alpha,d}(n_\alpha(\gamma),m_\alpha(\gamma)), \Lambda\Big) \\[10pt]
& \R\big( (\G_a)^{\bar n(\gamma)} \times (\G_m)^{\bar m(\gamma)}, \Lambda\big) \ \ol \ \Big( \hskip-3mm \displaystyle \bigotimes_{\alpha \in [\Delta/\phi]}^{\mathrm{L}} \hskip-2mm e_{\psi_\alpha} \R(\Y_{q_\alpha,d}(n_\alpha(\gamma),m_\alpha(\gamma)), \Lambda) \Big) \end{array}$}

\noindent Since $\psi$ is a regular character, every restriction $\psi_\alpha$ is a non trivial character of $\F_{q_\alpha}$. It follows from the lemma \ref{brl} that the complex $e_\psi \R(\db \Y_\gamma, \Lambda)$ is quasi-isomorphic to zero as soon as one of the $n_\alpha(\gamma)$ is different from zero. But for $\gamma \neq (1,1,\ldots,1)$, the set $I(\gamma)\setminus J(\gamma)$ is non-empty, and if $i_0$ denotes its largest element, then $\gamma^{i_0} = 1$ and $\gamma^{i_0-1}$ is a simple reflection $s_\alpha$ for some positive simple root $\alpha \in \Delta$. Therefore, $n_{-w_0(\alpha)}(\gamma) \neq 0$ and $e_\psi \R(\Y_\gamma'', \Lambda) \simeq 0$. 

\sk

If $\gamma = (1,1, \ldots, 1)$, then the invariants $\bar m(\gamma), \bar n(\gamma)$ and $n_\alpha(\gamma)$ are all equal to zero, and by lemma \ref{brl}, one has

\vskip -6pt\centers{$ \begin{array}{r@{\, \ \simeq \, \ }l} e_\psi \R(\db \Y_\gamma, \Lambda) & \Big( \hskip-3mm \displaystyle \bigotimes_{\alpha \in [\Delta/\phi]}^{\mathrm{L}} \hskip-2mm e_{\psi_\alpha} \R(\Y_{q_\alpha,d}(0,m_\alpha(\gamma)),\Lambda) \Big) {\ol} _{\Lambda (\mathbf{\bbmu}_d)^{m(\gamma)}} \Lambda \T^{wF}  \\[12pt] &
 \Big( \hskip-3mm \displaystyle \bigotimes_{\alpha \in [\Delta/\phi]}^{\mathrm{L}} \hskip-2mm \Lambda_{\psi_\alpha} (\mathbf{\bbmu}_d)^{m_\alpha(\gamma)} [-m_\alpha(\gamma)]\big) {\ol} _{\Lambda (\mathbf{\bbmu}_d)^{m(\gamma)}} \Lambda \T^{wF} \\[12pt] &  \Lambda_\psi (\mathbf{\bbmu}_d)^{m(\gamma)}[-m(\gamma)] {\ol} _{\Lambda (\mathbf{\bbmu}_d)^{m(\gamma)}} \Lambda \T^{wF} \\[10pt]
 e_\psi \R(\db \Y_\gamma, \Lambda) & \Lambda_\psi \T^{wF}[-m(\gamma)] 
\end{array}$}

\noindent which give the conclusion since $m(\gamma) = \ell(w)$ for this particular subexpression.
 \end{proof}

By combining the last proposition and the filtration property (see lemma \ref{filt}), we can finally state the second part of the result expected in \cite[Conjecture 2.7]{BR}, which finishes the proof of theorem \ref{br}.

\begin{cor}\label{cor2}Let $\psi : U \longrightarrow \Lambda^\times$ be a $G$-regular linear character. Then

\centers{$ e_\psi \R (\Yo,\Lambda) \, \simeq  \, \Lambda_\psi \T^{wF} [-\ell(w)] $}

\noindent in the derived category $\mathcal{D}^b(\Lambda U$-$\mathrm{mod}$-$\Lambda \T^{wF})$.
\end{cor}

\begin{rmk}\ With the results of \cite{BR}, we can actually compute the cohomology of each variety $\db \Y_\gamma$, and not only the $\psi$-isotypic part. However, except in some very special cases, it does not lead to a complete description of the complex for the variety $\db \Yo$ since it involves highly non-trivial triangles in the category $\mathcal{D}^b(\Lambda U$-$\mathrm{mod}$-$\Lambda \T^{wF})$.\end{rmk}

\end{document}